\documentclass{amsart}
\usepackage[hidelinks]{hyperref}
	\usepackage[usenames,dvipsnames,svgnames,table]{xcolor}

\usepackage[utf8]{inputenc}
\usepackage{amsmath}
\usepackage{amsthm}
\usepackage{amssymb}
\usepackage{mathtools}
\usepackage{enumitem}
\usepackage[numbers]{natbib}
\usepackage{prettyref}

\def\cexp#1#2{{#1}^{\odot #2}}

\theoremstyle{plain}
\makeatletter
\newtheorem*{thm@P}{\pipp@}
\newenvironment{thmC}[1]{\def\pipp@{#1}\begin{thm@P}}{\end{thm@P}}
\makeatother

\newtheorem{thm}{Theorem}[section]
\newtheorem*{thm*}{Theorem}
\newrefformat{thm}{Theorem~\ref{#1}} 
\let\oldnewtheorem\newtheorem 
\renewcommand{\newtheorem}[2]{\oldnewtheorem{#1}[thm]{#2} 
	\newrefformat{#1}{#2~\ref{##1}}} 

\newtheorem{prop}{Proposition}
\newtheorem{cor}{Corollary}

\newtheorem{lem}{Lemma}

\newtheorem{fact}{Fact}

\theoremstyle{definition}
\newtheorem{defn}{Definition}

\oldnewtheorem{claim}{Claim}[thm]

\newrefformat{claim}{Claim~\ref{#1}}
\theoremstyle{remark}
\newtheorem{rem}{Remark}

\newcommand{\R}{\mathbb R}
\newcommand{\N}{\mathbb N}

\newcommand{\K}{\mathbb K}
\newcommand{\no}{\mathbf {No}}
\newcommand{\Z}{\mathbb Z}

\newcommand{\m}{\mathfrak m}
\newcommand{\x}{\mathbf{x}}
\newcommand{\nin}{\not\in}

\newcommand{\class}[1]{\left| #1/ \! \asymp \right|}

\def\eps{\varepsilon}
\DeclareMathOperator{\suchthat}{\,:\,}

\def\on{\mathbf{On}}

\def\M{\mathfrak{M}}
\def\G{\mathfrak{G}}
\def\T{\mathbb{T}}
\def\E{\mathbb{E}}

\def\Sk{Sk}
\title{Asymptotic analysis of Skolem's exponential functions}
\author{Alessandro Berarducci}
\author{Marcello Mamino}
\address{Alessandro Berarducci and Marcello Mamino, Università di Pisa, Dipartimento di Matematica, Largo Bruno Pontecorvo 5, 56127 Pisa, Italy}
\date{18 November 2019, revised 24 March 2020}
\thanks{Both authors were supported by the Italian research project PRIN 2017, ``Mathematical logic: models, sets, computability'', Prot. 2017NWTM8RPRIN}
\begin{document}

\begin{abstract}
	Skolem (1956) studied the germs at infinity of the smallest class of real valued functions on the positive real line containing the constant $1$, the identity function $\x$, and such that whenever $f$ and $g$ are in the set, $f+g,fg$ and $f^g$ are in the set. This set of germs is well ordered and Skolem conjectured that its order type is epsilon-zero. Van den Dries and Levitz (1984) computed the order type of the fragment below $2^{2^\x}$. Here we prove that the set of asymptotic classes within any archimedean class of Skolem functions has order type $\omega$. As a consequence we obtain, for each positive integer $n$, an upper bound for the fragment below $2^{n^\x}$. We deduce an epsilon-zero upper bound for the fragment below $2^{\x^\x}$, improving the previous epsilon-omega bound by Levitz (1978). A novel feature of our approach is the use of Conway's surreal number for asymptotic calculations. 
\end{abstract}
\maketitle

\tableofcontents

\section{Skolem problem}
Let $\Sk$ be the smallest set of functions $f:\R^{>0}\to \R^{>0}$ containing the constant function $1$ and the identity function $\x$, and such that if $f,g\in \Sk$, then also $f+g,fg$ and $f^g$ are in $\Sk$. A Skolem function is a function belonging to $\Sk$. Each Skolem function restricts to a function $f:\N^{>0}\to \N^{>0}$ from positive integers to positive integers and it is determined by its restriction. 

We order $\Sk$ by $f< g$ if $f(x)<g(x)$ for all large enough $x$ in $\R$ (or equivalently in $\N$). This defines a total order. 
Indeed \citet{Hardy1910} established the corresponding result for a larger class of functions. The totality of the order also follows from the fact that the structure $\R_{\exp} = (\R,<,+,\cdot,\exp)$  is o-minimal \cite{Wilkie1996} and the Skolem functions are definable in $\R_{\exp}$. 

In this paper we study the order type of $\Sk$ and its fragments.   
\citet{Skolem1956} conjectured that $(\Sk,<)$ is a well order and its order type is  $\eps_{0} = \sup\{\omega, \omega^\omega, \omega^{\omega^\omega},\ldots\}$ (the least ordinal $\eps$ such that $\eps= \omega^\eps$). He also exhibited a well ordered subset of order type $\eps_{0}$, namely the subset generated from $1$ and $\x$ using the operations $+,\cdot$ and exponentiation $g\mapsto \x^g$ with base $\x$. 
\citet{Ehrenfeucht73}, using the tree theorem of \citet{Kruskal1960}, proved that $\Sk$ is indeed well ordered. 
\citet{Levitz1978} showed that its order type is at most equal to the smallest critical epsilon-number (the least ordinal $\alpha$  such that $\alpha = \eps_{\alpha}$).  This improves the earlier bound $\Gamma_0$ established by \citet{Schmidt1978}, where $\Gamma_0$ is the Feferman-Sch\"utte ordinal. 

Given a well ordered set $X$, we write $|X|$ for the order type of $X$. If $f\in \Sk$, we let $|f|$ be the order type of the set of Skolem functions less than $f$. The Skolem functions $<2^\x$ coincide with the non-zero polynomial functions with coefficients in $\N$, so $|2^\x| = |\omega^\omega|$. 

In \cite{Levitz1978} Levitz introduced the following definition:  a
regular function is a Skolem function $g$ such that for every Skolem
function $f<g$, one has  $f^\x < g$. The first regular functions are $g_0=
2$ and $g_1= 2^{2^\x}$ and it is not difficult to show that the regular
functions $<2^{\x^\x}$ are exactly the functions of the form $2^{n^\x}$
with $2\leq n \in \N$. Levitz proved that  $|g_{1+\alpha}| \leq
\eps_{\alpha}$, where  $(g_{\alpha})_\alpha$ is a transfinite enumeration
of the regular functions, and $ (\eps_\alpha)_\alpha$ is an enumeration of
the epsilon numbers (i.e.\ the ordinals $\eps$ sastisfying $\eps = \omega^\eps$). 
Levitz's result then yields $|2^{2^\x}| \leq \eps_0$, $2^{3^\x} \leq \eps_1$ and $|2^{\x^\x}| \leq \eps_\omega$ (since $g_{1+\omega}= g_\omega = 2^{\x^\x}$). 

In \cite{Dries1984} van den Dries and Levitz made a dramatic improvement on Levitz's bound on $g_1$ by showing that $|2^{2^{x}}| = \omega^{\omega^{\omega}}$. 
Here we prove the following bound on the fragments determined by the first $\omega$ regular functions. Let $\omega_0 = 1$ and $\omega_{n+1} = \omega^{\omega_n}$ for $n\in \N$. We have:
\begin{thmC}{\prettyref{thm:2nx}}
	$|2^{n^\x}| \leq \omega_{n+1}$ for $n\geq 1$.  In particular $|2^{3^\x}| \leq   \omega_4 = \omega^{\omega^{\omega^\omega}}$. 
\end{thmC}
\noindent \prettyref{thm:2nx} should be compared with Levit'z bound $|2^{3^\x}| \leq \eps_1$. 
As a consequence we obtain the following upper bound on $2^{\x^\x}$, which improves Levitz's $\eps_\omega$ bound. 
\begin{thmC}{\prettyref{thm:2xx}}
	$|2^{\x^\x}| \leq \eps_0$. 	
\end{thmC} 

A novel feature of our approach is the use of Conway's surreal
numbers~\cite{Conway76} for asymptotic calculations, justified by the fact
that the Skolem functions can be embedded in the exponential field of
surreal numbers,
that is, one can associate a surreal number to each Skolem function
preserving the field operations, exponentiation, and ordering.
Our main result is as follows.
\begin{thmC}{\prettyref{thm:main}}
Let $c\geq 1$ be a surreal number and let $Q$ be a Skolem function. 
The set of real numbers $r\in \R$ such that there is a Skolem function~$h$
satisfying $(h/Q)^c = r + o(1)$ has no accumulation points in $\R$.
\end{thmC}

The case $c=1$ of the theorem says that, if we fix a Skolem function $Q(x)$, the set of real numbers $r$ such that there is a Skolem function $h(x)$ with $\lim_{x\to \infty}h(x)/Q(x) = r$, has no accumulation points in $\R$. This special case is sufficient to obtain the bounds above, and also yields a different proof of the upper bound in~\cite{Dries1984}. It turns out that, for technical reasons, we need to consider the general case $c\geq 1$ in order to prove the special case $c=1$. 

In the preliminary part of the paper, we prove a result
concerning the order type of the set of finite sums $\sum A$ of a well
ordered set $A$ of positive elements of an ordered group
(\prettyref{thm:finite-sums}). Unlike the known bounds by
\citet{Carruth1942} and other authors, our bound takes into account the
set of archimedean classes of $A$. 

The equality of two Skolem functions (given the defining expressions) is
decidable \citep{Richardson1969}, but it is an open problem whether the
order $<$ is decidable. \citet{Gurevlc1986} established the decidability
of $<$ below $2^{\x^2}$ and showed that the decidability of $<$ below
$2^{2^\x}$ is Turing equivalent to the decidability of the equality of two
``exponential constants'', where the exponential constants are the
elements in the smallest subset  $\E^+ \subset \R$ containing $1$ and
closed under addition, multiplication, division, and the real exponential
function. In \cite{Dries1984} van den Dries and Levitz proved that if the
quotient $f/g$ of two Skolem functions smaller than $2^{2^\x}$ tends to a
limit in $\R$, then the limit is in $\E^+$. They announced that the result
could be extended to the whole class of Skolem functions using the work of
\cite{Dahn1984a}, where a version of the field of transseries made its first
appearance. In the last part of the paper we give a proof of these facts
using surreal numbers. 

\section{Asymptotic relations}
Given $f,g$ in an ordered abelian group, we write $f\preceq g$
if $|f|\leq n|g|$ for some $n\in\N$. We say in this case that $f$ is {\bf dominated} by $g$. If both $f\preceq g$ and $g\preceq f$ hold, we say that $f$ and $g$ belong to the same {\bf archimedean class}, and we write  $f\asymp g$. We say that $f$ is {\bf strictly dominated} by $g$ if we have both $f\preceq g$ and $f\not\asymp g$; we write $f\prec g$ to express this relation.  We define $f\sim g$ as $f-g \prec f$ and we say in this case that $f$ is {\bf asymptotic} to $g$. Notice that $\sim$ is a symmetric relation. Indeed assume $f-g \prec f$ and let us prove that $f-g\prec g$. This is clear if $f\preceq g$. On the other hand if $g \prec f$, then clearly $f-g \asymp f$, contradicting the assumption. 

We write $f = o(g)$ if $f\prec g$ and $f= O(g)$ if $f\preceq g$. 

The set of germs at $+\infty$ of the Skolem functions generates an ordered field by the results of \cite{Hardy1910} or \cite{Wilkie1996} cited in the introduction, so we can use the above notations for the Skolem functions. By the cited results, the quotient $f(x)/g(x)$ of two Skolem functions tends to a limit in $\R\cup \{+ \infty\}$ for $x\to +\infty$. We then have $f\prec g$ if $f/g$ tends to $0$; 
$f\sim g$ if $f(x)/g(x)$ tends to $1$; and $f\asymp g$ if  $f/g$ tends to a non-zero limit in $\R$. Note that $f\asymp g$ if and only if there is a non-zero real number $r$ such that $f \sim rg$. We will prove as a special case of \prettyref{thm:main} that, if we fix $g$ and let $f$ vary in $\Sk$, then the corresponding real $r$ ranges in a subset of $\R$ without accumulation points.

\section{Ordinal arithmetic}
Let $\on$ be the class of all ordinal numbers.	Given $\alpha \in \on$ and
$\beta \in \on$, we write $\alpha+ \beta$ and $\alpha\beta$ (or sometimes
$\alpha \cdot \beta$) for the ordinal sum and product of the given
ordinals, and $\alpha^\beta$ for the ordinal exponentiation. We identify
each ordinal with the set of its predecessor and we denote by $\omega$ the
first infinite ordinal, which can also be thought as the set of all finite
ordinals, i.e.\ the set of natural numbers $\N$. 

\begin{defn} Given a sequence $(\alpha_i)_i$ of ordinals, we define inductively: 
\begin{enumerate}
	\item $\sum_{i<0} \alpha_i = 0$;
	\item 
	$\sum_{i<\beta+1}\alpha_i = \sum_{i<\beta}\alpha_i + \alpha_\beta$;
	\item 
	$\sum_{i<\lambda} \alpha_i = \sup_{\beta<\lambda} \sum_{i<\beta} \alpha_i$ for $\lambda$ a limit ordinal.
\end{enumerate}  
\end{defn}
We recall that every ordinal $\alpha$ can be written in a unique way in the form $\alpha = \sum_{i<n}\omega^{\gamma_i} n_i$ where $n\in \N$, $(\gamma_i)_{i<n}$ is a decreasing sequence of ordinals, and $n_i\in \N^{>0}$ for each $i<n$.  This is called the Cantor normal form of $\alpha$. 

We write $\alpha \oplus \beta$ and $\alpha \odot \beta$ for the Hessenberg sum and product \cite{Sierpinski1958}. We recall the definitions below. 
\begin{defn}
	Given $\alpha\in \on$ and $\beta\in \on$, we can find $k\in \N$ and a decreasing finite sequence of ordinals $(\gamma_i)_{i<k}$ such that 
	$\alpha = \sum_{i<k}\omega^{\gamma_i}m_i$ and $\beta =\sum_{i<k}\omega^{\gamma_i}n_i$ with $m_i,n_i < \omega$ (possibly zero). We define 
	$$\alpha \oplus \beta= \sum_{i<k} \omega^{\gamma_i} (m_i+n_i).$$
\end{defn}
\begin{defn} 
	If $\alpha = \sum_{i<k}\omega^{\alpha_i}m_i$ and $\beta = \sum_{i<l}\omega^{\beta_j}n_j$ are two ordinals in Cantor normal form, their Hessenberg product is defined as 
	$$\alpha \odot \beta = \bigoplus\limits_{i<k,j<l}\omega^{\alpha_i \oplus \beta_j} m_i n_j.$$
\end{defn}

We shall need transfinite iterations of the Hessenberg sum and product.
\begin{defn} Given a sequence of ordinals $(\alpha_i)_{i}$ we define inductively:
	\begin{enumerate}
		\item $\bigoplus_{i<0} \alpha_i = 0$,
		\item  $\bigoplus_{i<\beta+1} \alpha_i = (\bigoplus_{i<\beta} \alpha_i) \oplus  \alpha_{\beta}$
		\item  $\bigoplus_{i<\lambda} \alpha_i = \sup_{\beta < \lambda} \bigoplus_{i<\beta} \alpha_i$ for $\lambda$ limit. 
	\end{enumerate}    
\end{defn}
The paper \citep{Lipparini2018} contains some comparison results between $\sum_{i<\beta}$ and $\bigoplus_{i<\beta}$. 
Similarly we define the transfinite iteration of the Hessenberg product.
\begin{defn}  Given a sequence of ordinals $(\alpha_i)_{i}$ we define inductively:
	\begin{enumerate}
		\item $\bigodot_{i<0} \alpha_i = 1$,
		\item  $\bigodot_{i<\beta+1} \alpha_i = (\bigodot_{i<\beta} \alpha_i) \odot  \alpha_{\beta}$
		\item  $\bigodot_{i<\lambda} \alpha_i = \sup_{\beta < \lambda} \bigodot_{i<\beta} \alpha_i$ for $\lambda$ limit. 
	\end{enumerate}   
\end{defn}

\begin{defn} Given two ordinals $\alpha$ and $\beta$ we define $\cexp{\alpha}{\beta} = \bigodot_{i<\beta} \alpha$.  
\end{defn}

\begin{prop} \label{prop:powers-of-n}
	If $n<\omega$, then $\cexp{n}{\gamma} = n^{\gamma}$ for every $\gamma\in \on$. 
\end{prop}
\begin{proof} We can write $\gamma = \omega \beta + k$ with $\beta\in \on$ and $k<\omega$. Since $n<\omega$, $n^\omega = \omega$, and therefore $n^{\omega \beta + k} = \omega^\beta n^k$. On the other hand by \cite[Lemma 3.6]{Altman2017} we have
	$\cexp{n}{\omega \beta + k} =\omega^\beta n^k 
	=n^{\omega \beta + k}$, thus concluding the proof.  
\end{proof}

\begin{lem} \label{lem:natural-sum} If $\alpha \geq \beta$, then 
	$\alpha\oplus\beta\leq \alpha+\beta 2$. 
\end{lem}
\begin{proof}
	We can assume $\beta>0$. Let $\alpha=\sum_{i<k}\omega^{\alpha_{i}}m_{i}$ and $\beta=\sum_{j<l}\omega^{\beta_{j}}n_{j}$ be Cantor normal forms. For some $i_0 \leq k$, $\alpha\oplus\beta$ has the form $\sum_{i< i_{0}}\omega^{\alpha_{i}}m_{i}+\omega^{\beta_{0}}n_{0}+\rho$ with 
	$\rho<\omega^{\beta_{0}}\leq \beta$. Since $\omega^{\beta_0} n_0 \leq \beta$, we obtain $\alpha \oplus \beta \leq \alpha + \beta + \beta = \alpha + \beta 2$. 
\end{proof}
\begin{lem}  \label{lem:bound-on-nat-sum}
	$\alpha \beta \; \leq \; \bigoplus_{i<\beta}\alpha \; \leq \; \alpha 2\beta$. 
\end{lem}
\begin{proof}
	By induction on $\beta$ based on \prettyref{lem:natural-sum}. The case when $\beta$ is zero or a limit ordinal follows at once from the induction hypothesis. If $\beta = \gamma+1$, then $\alpha(\gamma+1) \leq \bigoplus_{i<\gamma+1} \alpha = (\bigoplus_{i<\gamma} \alpha) \oplus \alpha \leq \alpha 2 \gamma \oplus \alpha \leq \alpha 2 \gamma + \alpha 2 = \alpha2 (\gamma+1)$, where we used \prettyref{lem:natural-sum} and the induction hypothesis. 
\end{proof}

\begin{cor}\label{cor:iterated-sum}
	If $\lambda$ is limit, then $\bigoplus_{i<\lambda}\alpha =  \alpha\lambda$.
\end{cor}

\begin{proof}
	If $\beta$ is limit, then $2\beta = \beta$, so  we can conclude by \prettyref{lem:bound-on-nat-sum}. 
\end{proof}

Let $\alpha$ be an ordinal. We say that $\alpha$ is {\bf additively closed} if the sum of two ordinals less then $\alpha$ is less than $\alpha$. Similarly, $\alpha$ is {\bf multiplicatively closed} if the product of two ordinals less than $\alpha$ is less than $\alpha$. We obtain an equivalent definition using the Hessenberg sum and product. The additively closed ordinals $>0$ are the ordinals of the form $\omega^\delta$ for some $\delta$; the multiplicatively closed ordinals $>1$ are the ordinals of the form $\omega^{\omega^\delta}$ for some $\delta$ \cite{Sierpinski1958}. 

\begin{prop}\label{prop:cexp}
	If $\alpha\in \on$ and $\lambda$ is a limit ordinal, then $\cexp{\alpha}{\lambda} = \alpha^\lambda$. Moreover $\alpha^\lambda$ is additively closed. 
\end{prop}

\begin{proof} 
	The case $\alpha<\omega$ follows from \prettyref{prop:powers-of-n}.
	Assume $\alpha \geq \omega$ and consider first the special case $\alpha = \omega^\gamma$. For every $\beta$ it is easy to verify by induction that  
	$\cexp{({\omega^\gamma})}{\beta} 
	=
	\bigodot_{i<\beta} \omega^{\gamma} 
	=
	\omega^{\bigoplus_{i<\beta} \gamma}$. Now take $\beta = \lambda$. 
	Since $\lambda$ is limit, by \prettyref{cor:iterated-sum}, $\bigoplus_{i<\lambda} \gamma = \gamma\lambda$, so 
	$\cexp{({\omega^\gamma})}{\lambda} 
	=
	\omega^{\gamma \lambda}$. 
	
	For a general $\alpha \geq \omega$, let $\delta>0$ be such that $\omega^\delta \leq \alpha < \omega^{\delta+1}$. 
	Since $\lambda$ is limit, $(\delta+1)\lambda = \delta \lambda$. 
	The result now follows from the inequalities $\cexp{\alpha}{\lambda} \leq \cexp{(\omega^{\delta+1})}{\lambda} = \omega^{(\delta+1)\lambda} = \omega^{\delta \lambda} \leq \alpha^\lambda \leq \cexp{\alpha}{\lambda}$. 
\end{proof}

\begin{cor}\label{cor:cexp}
	For $\beta\in \on$, let $\beta =\lambda + k$ with $\lambda$ a limit ordinal or zero and $k<\omega$. Then $\cexp{\alpha}{\beta} = \alpha^\lambda \odot \cexp{\alpha}{k}$. 
\end{cor}

\section{Well ordered subsets of ordered groups}
The Hessenberg sum and product can be characterized as follows. Consider disjoint well ordered sets $A$ and $B$ of order type $\alpha$ and $\beta$ respectively. By \citep{Carruth1942} or \citep{DeJongh1977} the Hessenberg sum $\alpha\oplus \beta$ is the sup of all ordinals $\gamma$ such that one can extend the given partial order on $A\cup B$ to a total order of order type $\gamma$; the Hessenberg product $\alpha \odot \beta$ is the sup of all ordinals $\gamma$ such that one can extend the componentwise partial order on $A\times B$ to a total order of order type $\gamma$. By the cited papers, the sups are achieved.   
An immediate consequence of the above characterization is the following: 
\begin{fact}\label{fact:binary}
	Let $X = (X,<)$ be a totally ordered set and let $A, B\subseteq X$ be well ordered subsets. We have:
	\begin{enumerate}
		\item $A\cup B$ is well ordered and $|A\cup  B|\leq |A|\oplus |B|$ 
		\item Let $f:X\times X \to X$ be a binary function which is weakly increasing in both arguments and let $f(A,B) := \{f(a,b):a\in A, b\in B\}$. Then $f(A,B)$ is well ordered and $|f(A,B)|\leq |A| \odot |B|$.
	\end{enumerate}
\end{fact}
Given two sets $A$ and $B$ of Skolem functions we write: $A+B$ for the set of all sums $f+g$ with $f\in A$ and $g\in B$; $AB$ for the set of all products $fg$ with $f\in A$ and $g\in B$;  $A^B$ for the set of all functions of the form $f^g$ with $f\in A$ and $g\in B$. 
We write $A/\!\asymp$ for the ordered set of all $\asymp$ classes of elements of $A$, and similarly for $A/\!\sim$. 

\begin{cor} \label{cor:binary-bounds}Let $A$ and $B$ be sets of Skolem functions. Then:
	\begin{enumerate}
		\item $A\cup B$ has of order type $\leq |A| \oplus |B|$. 
		\item  $A+B,\; AB$ and $A^B$ have order type $\leq |A|\odot |B|$. 
		\item $AB/\!\asymp$ has order type $\leq |A/\! \asymp| \odot |B/\! \asymp |$. 
	\end{enumerate}
\end{cor} 
\begin{proof} The first two points are immediate from \prettyref{fact:binary}. 
	To prove point (3) we use again \prettyref{fact:binary} together with the observation that the $\asymp$-class of $fg$ depends only on the respective $\asymp$-classes of $f$ and $g$, and this dependence is weakly increasing in both arguments. 
\end{proof}

Given a subset $A \subset \Sk$, we write $\sum A$ for the set of finite non-empty sums of elements from $A$. 
We want to give an upper bound on $|\sum A|$. The definition of $\sum A$ can be given more generally for a subset $A$ of an ordered abelian group $G$, so it is convenient to work in this context. If $A$ is a well ordered subset of $G^{> 0}$, $\sum A$ is well ordered and \citet{Carruth1942} gave an upper bound on its order type in terms of the order type of $A$. In \prettyref{thm:finite-sums} we obtain a different bound which takes  into account the set of archimedean classes of $A$.

\begin{lem}
	\label{lem:different-classes}Let $(G,+,<)$ be an ordered abelian
	group and let $A\subseteq G^{>0}$ be a well ordered subset of order
	type $\alpha$.  Suppose all the elements of $A$ belong to distinct
	archimedean classes. Then the order type of $\sum A$ is $\leq\omega^{\alpha}$. 
\end{lem}
\begin{proof}
	Let $(a_{i}:i<\alpha)$ be an increasing enumeration of $A$. Let
	$x\in \sum A$. Then $x$ can be written uniquely in the form $x=\sum_{i<\alpha}a_{i}n_{i}$
	where $n_{i}\in\N$ and $n_{i}=0$ for all but finitely many $i$.
	We associate to $x$ the ordinal $\bigoplus_{i<\alpha}\omega^{i}n_{i}$.
	This defines an increasing map from $\sum A$ to $\omega^{\alpha}$
	yielding the desired result. 
\end{proof}

\begin{lem}\label{lemma:archimedean}
	\label{lem:archimedean-case}Let $(G,+,<)$ be an ordered abelian
	group and let $A\subseteq G^{>0}$ be a well ordered subset of order
	type $\alpha \geq 2$. Suppose all the elements of $A$ belong to the same
	archimedean class. Then $$
	|\sum A| \leq  \alpha^{\omega}.$$ (If $|A| \leq 1$, clearly $|\sum A| \leq \omega$.) 
\end{lem}

\begin{proof} Let $b\in\sum A$, let $(\sum A)^{<b}$ be the set of elements less than $b$ in $\sum A$. Since all elements of $A$ belong to the archimedean class of its least element, there exists $m\in \N$, depending on $b$, such that every element of $(\sum A)^{<b}$ is the sum of at most $m$ elements of $A$. By induction on $i\leq m$ using \prettyref{cor:binary-bounds}, the set of sums of $i$ elements of $A$ has order type 
	$\leq \cexp{\alpha}{i}$. By the same corollary it then follows by induction on $m$ that 
	$|(\sum A)^{<b}| \leq \bigoplus_{i=1}^m \cexp{\alpha}{i}$. 
	Now for each $i\leq m$, 
	$\cexp{\alpha}{i} < \cexp{\alpha}{\omega}$ and $\cexp{\alpha}{\omega} = \alpha^\omega$ is additively closed (\prettyref{prop:cexp}). It follows that 
	$|(\sum A)^{<b}| < \alpha^\omega$. 
	Since this holds for every $b\in \sum A$, we can conclude that $|\sum A| \leq \alpha^\omega$. 
\end{proof}

\begin{thm}\label{thm:finite-sums}
	Let $(G,+,<)$ be an ordered abelian group, let $A\subseteq G^{>0}$
	be a well ordered set of order type $\alpha \geq 2$ and let $\beta = |A/\! \asymp|$ be the order type of the set of archimedean classes of $A$. Then the order type of $\sum A$ is $\leq \cexp{(\alpha^\omega)}{\beta}$. 
\end{thm} 
\begin{proof}
	Let $B\subseteq A$ be a set of representatives for the archimedean
	classes of $A$ and let $(b_{i}\suchthat i<\beta)$ be an increasing
	enumeration of $B$. We reason by induction on $\beta$. The case $\beta=1$ is \prettyref{lemma:archimedean}. 
	
	Case $\beta$ limit. For $b\in G^{>0}$, let $A^{\preceq b}$ be the subset of $A$ consisting of the elements $\preceq b$ and let $A^{\asymp b}$ be the set of elements of $A$ which are  $\asymp b$. 
	Then $\sum A = \bigcup_{\gamma < \beta} \sum(A^{\preceq b_\gamma})$. The sets in the union are pairwise initial segments of one another. It follows that the order type of the union is the sup of the respective order types. By induction 
	$|\sum A| \leq \sup_{\gamma <\beta} \cexp{(\alpha^\omega)}{\gamma}
	=
	 \cexp{(\alpha^\omega)}{\beta}$. 
	
	Case $\beta = \gamma + 1$. We have  $\sum A = \sum(A^{\prec b_\gamma}) + \sum (A^{\asymp b_\gamma})$. By the induction hypothesis
	 $|\sum(A^{\prec b_\gamma})| \leq \cexp{(\alpha^\omega)}{\gamma}$.
	  The elements of $A^{\asymp b_\gamma}$ live in a single archimedean class, so $|\sum (A^{\asymp b_\gamma})| \leq \alpha^\omega$. It follows that $|\sum A| 
	  \leq
	   \cexp{(\alpha^\omega)}{\gamma}
	   \odot 
	   \alpha^\omega 
	  =
	  \cexp{(\alpha^\omega)}{(\gamma+1)}
	  $. 
\end{proof}

We define a sequence of countable ordinals as follows.
\begin{defn}
	Let $\omega_0 = 1$ and, inductively,  $\omega_{n+1} = \omega^{\omega_n}$. 
\end{defn}

\begin{rem}
	For all $n\in \N$, $\omega_n$ is multiplicatively  closed.  
\end{rem}
\begin{proof} Clearly the product of two ordinals $<1$ is $<1$, so the property holds for $n=0$. 
	For $n\geq 1$, $\omega_n$ has the form $\omega^{\omega^\delta}$ (e.g. $\omega_1 = \omega = \omega^{\omega^0}$ and $\omega_2 = \omega^\omega = \omega^{\omega^1}$), so it is multiplicatively closed. 
\end{proof}

For our applications we need the following lemma. 
\begin{lem}\label{lem:property-of-omega-n} Let $2 \leq n < \omega$. 
	If $\alpha<\omega_{n+1}$ and $\beta < \omega_{n}$, then $\cexp{(\alpha^\omega)}{\beta} < \omega_{n+1}$.
\end{lem}
\begin{proof} 
	We can write $\beta = \lambda + k$ where $\lambda$ is a limit ordinal or zero and $k<\omega$. By \prettyref{cor:cexp} we have $\cexp {(\alpha^\omega)}{\beta} = \alpha^{\omega \lambda} \odot \cexp{(\alpha^\omega)}{k}$.  
	Since $\alpha < \omega_{n+1}= \omega^{\omega_n}$ and $\omega_n$ is a limit ordinal, there is some $\gamma<\omega_n$ such that $\alpha \leq \omega^\gamma$. Since $\omega$ and $\lambda$ are $<\omega_n$ and $\omega_n$ is multiplicatively closed, we have $\omega \lambda <\omega_n$, hence $\alpha^{\omega \lambda} < \omega_{n+1}$. Similarly, $\cexp{(\alpha^\omega)}{k}= \cexp{\alpha^\omega}{k}< \omega_{n+1}$. Now since $\omega_{n+1}$ is multiplicatively closed, $\alpha^{\omega \lambda} \odot (\cexp{\alpha^\omega}{k})<\omega_{n+1}$, as desired. 
\end{proof}

\begin{cor}\label{cor:bounds-on-sums}
	Let $(G,+,<)$ be an ordered abelian group, let $A\subseteq G^{>0}$
	be a well ordered set of order type $<\omega_{n+1}$ whose set of archimedean classes has order type $<\omega_n$. Then the order type of $\sum A$ is $<\omega_{n+1}$ and its set of archimedean classes has order type $<\omega_n$. 
\end{cor}
\begin{proof}
	By \prettyref{lem:property-of-omega-n} and \prettyref{thm:finite-sums}, together with the observation that the set of archimedean classes does not changes under taking finite sums. 
\end{proof}

Another interesting bound on $|\sum A|$ is contained in \cite{DriesE2001}: if $|A| \leq \alpha$, then $|\sum A| \leq \omega^{\omega \alpha}$. For our purposes we need the bound in \prettyref{cor:bounds-on-sums} which takes into account also the order type of the archimedean classes of $A$. Note that both bounds imply that if $|A|<\eps_{0}$, then $|\sum A| < \eps_{0}$. 

\section{Generalized power series}\label{sec:series}
Given an ordered field $K$, a multiplicative subgroup $\M$ of $K^{>0}$ is called a group of {\bf monomials} if for each non-zero element $x$ of $K$ there is one and only one $\m\in \M$ such that $x\asymp \m$. We assume some familiarity with Hahn's field $\R((\M))$ of generalized power series \cite{Hahn1907}, but we recall a few definitions. An element of $\R((\M))$ is a formal sum $f= \sum_{i<\alpha} \m_i r_i$ where $\alpha$ is an ordinal, $(\m_i)_{i<\alpha}$ is a decreasing sequence in $\M$, and $0\neq r_i\in \R$ for each $i<\alpha$. We say that $\{\m_i \mid i<\alpha\}$ is the {\bf support} of the series $\sum_{i<\alpha} \m_i r_i$. The sum and product of generalized series  is defined in the obvious way. We order $\R((\M))$ by  $f = \sum_{i<\alpha} \m_i r_i > 0 \iff r_0 > 0$. This makes $\R((\M))$ into an ordered field with $\M$ as a group of monomials (where we identify $\m\in \M$ with $\m1\in \R((\M))$). 

A family $(f_i)_{i\in I}$ of elements of $\R((\M))$ is {\bf summable} if each monomial $\m\in \M$ is contained in the support of finitely many $f_i$ and there is no strictly increasing sequence $(\m_n)_{n\in \N}$ of monomials such that each $\m_n$ belongs to the support of some $f_i$. In this case $\sum_{i\in I}f_i\in \no$ is defined adding the coefficients of the corresponding monomials. 

To prove that $\R((\M))$ is a field, we write a non-zero element $x$ of $\R((\M))$ in the form $r\m(1+\eps)$ with $r\in \R^*, \m\in \M$ and $\eps \prec 1$ and observe that $x^{-1} = r^{-1}\m^{-1}\left(\sum_{n\in \N} (-1)^n \eps^n \right)$ where the summability of $(-1)^n \eps^n$ is ensured by Neumann's Lemma \cite{Neumann1949}. More generally Neumann's Lemma says that if $\eps$ is an infinitesimal element of $\R((\M))$ and $(r_n)_{n\in \N}$ is any sequence of real numbers, then $(r_n \eps^n)_{n\in \N}$ is summable, so we can evaluate the formal power series $P(X) = \sum_{n\in \N} r_n X^n$ at any infinitesimal element of $\R((\M))$ . 

Given $f$ and $g$ in $\R((\M))$ we say that $g$ is a {\bf truncation} of $f$ if $f = \sum_{i<\alpha} \m_i r_i$ and  $g = \sum_{i<\beta} \m_i r_i\in \R((\M))$ for some $\beta<\alpha$. If $\G\subseteq \M$ is a subset, we write $\R((\G))$ for the set of all $f\in \R((\M))$ whose support is contained in $\G$. 

\section{Surreal numbers} Conway's field $\no$ of surreal numbers \cite{Conway76,Gonshor1986} is an ordered real closed field extending the field $\R$ of real numbers and containing a copy of the ordinal numbers. In particular $\no$ is a proper class, and admits a group of monomials $\M\subset \no^{>0}$ which is itself a proper class. We can define generalized power series with monomials in $\M$ exactly as above, but we denote the resulting field as $\R((\M ))_\on$, where the subscript is meant to emphasize that, although $\M$ is a proper class, the support of a generic element $\sum_{i<\alpha} \m_i r_i$ of $\no$ is a set (because $\alpha$ is still assumed to be an ordinal). 
\citet{Conway76} showed that we can identity $\no$ with $\R((\M))_\on$, where the class $\M \subset \no$ of monomials is defined explicitly (it coincides with the image of Conway's omega-map). 

A surreal $x = \sum_{i<\alpha} \m_i r_i \in \R((\M))_\on$ is  {\bf purely infinite} if all monomials $\m_i$ in its support are $>1$ (hence infinite). We write $\no^\uparrow$ for the (non-unitary) ring of purely infinite surreals. We observe that 
every $x\in \no$ can be written in a unique way in the form $x = x^\uparrow + x^\circ + x^\downarrow$ where $x^\uparrow \in \no^\uparrow$, $x^\circ \in \R$ and $x^\downarrow \prec 1$. 
This yields a direct sum decompostion of $\R$-vector spaces
$$\no = \no^\uparrow +\R + o(1)$$ where $o(1)$ is the set of elements $\prec 1$.
\citet{Gonshor1986} defined an isomorphism of ordered groups $\exp:(\no,+,<)\to (\no^{>0},\cdot,<)$ extending the real exponential function and satisfying $\exp(x) \geq 1+x$ for all $x\in \no$ and $\exp(x) = \sum_{n\in \N}\frac{x^n}{n!}$ for $x\prec 1$ (we need $x\prec 1$ to ensure the summability of the series). Gonshor's $\exp$ is defined in such a way that $\exp(\no^\uparrow) = \M$, namely the monomials are the images of the purely infinite numbers. The stated properties are already sufficient to ensure that $\no$, with Gonshor's $\exp$, is a model of the elementary theory $T_{\exp}$ of the real exponential field $\R_{\exp}=(\R,<,+,\cdot,\exp)$; in other words $(\no,\exp)$ satisfies all the property which are true in $\R_{\exp}$ and are expressible by a first-order formula in the ring language and a symbol for the exponential function \cite{DriesE2001}. A discussion of these issues can also be found in \cite{Berarducci2018b}, where other fields of generalized power series admitting an exponential map resembling the surreal $\exp$ have been considered. 

As long as we are only interested in the elementary theory of $\no$ as an exponential field, both the choice of the monomials $\M\subset \no$ and the details of the definition of $\exp$ on $\no^\uparrow$ are not important. However they become important for summability issues and the properties of infinite sums, so we need to state a few more facts that are are needed in this paper (all of them can be found in \cite{Berarducci2019}). We denote by $\log:\no^{>0}\to \no$ the compositional inverse of $\exp$ and we also write $e^x$ for $\exp(x)$. 
It can be shown that if $x \prec 1$, then $\log(1+x) = \sum_{n=1}^\infty \frac{(-1)^{n+1}}{n}x^n$. An important fact, that depends on the choice of $\M\subset \no^{>0}$, is that $\omega$ is a monomial (where $\omega$ is the least infinite ordinal seen as a surreal). More generally, for each $n\in \N$, $\log_n(\omega)$ is an infinite monomial  \cite{Gonshor1986}, where $\log_0(\omega) = \omega$ and $\log_{n+1}(\omega) = \log(\log_n(\omega))$. This fact is used in \cite{Berarducci2019} to show that $\no$ contains an isomorphic copy of the field $\T$ of transseries  as an exponential field (the notation $\T$ is used in \cite{Aschenbrenner2015} and refers to the version of the transseries defined  \cite{DriesMM2001} under the name ``logarithmic exponential series). Moreover $\no$ admits a differential operator $\partial:\no\to \no$ extending the one on $\T$ \cite{Berarducci2018,Berarducci2019}. 
Since $\no^\uparrow$ is closed under multiplication by a real number, any real power $\m^r= e^{r \log(\m)}$ of a monomial is again a   monomial. Moreover, if $\m$ is an infinite monomial, $e^\m$ is again a monomial (because $\exp(\no^\uparrow) = \M$). 

From the equations $\no=\R((\M))_\on$ and $\M= e^{\no^\uparrow}$ it follows that every surreal can be written in a unique way in the form $\sum_{i<\alpha} e^{\gamma_i} a_i$ where 
$\alpha$ is an ordinal, $(\gamma_i)_{i<\alpha}$ is a decreasing sequence in $\no^\uparrow$ and $a_i \in \R^*$ (the empty sum is $0$). Following  \cite{Berarducci2018}, we call this representation {\bf Ressayre form}.

\section{Surreal expansions of Skolem functions}\label{sec:surreal-expansions}
Since the surreal numbers are a model of $T_{\exp}$ there is a unique map from $\Sk$ to  $\no$ sending the identity function $\x$ into $\omega$ and preserving $1,+,\cdot$ and the function $(a,b)\mapsto a^b$ where $a^b = e^{b\log (a)}$. Since $\omega$ is greater than any natural number, this map preserves the order, so it is an embedding of ordered semirings endowed with an exponential $(a,b)\mapsto a^b$. We identify a Skolem function $f= f(\x)\in \Sk$ with its image $f(\omega)\in \no$ under this embedding, and we define the {\bf Ressayre form of $f\in \Sk$} as the Ressayre form of the surreal number $f(\omega)$. 

If we identify the transseries $\T$ with a subfield of $\no$ (as in \cite{Berarducci2019}), it is easy to see that the image of the embedding of $\Sk$ in $\no$ is contained in $\T$, but we shall not need this fact.  

We can consider the Ressayre form of a Skolem function $f(x)$ as an asymptotic development for $x\to +\infty$. For example consider the Skolem function $(\x+1)^\x$ and identify $\x$ with $\omega\in \no$. To find its Ressayre form we write $(\x+1)^\x = e^{\x \log(1+\x)}$ and we expand $\log(1+\x)$ as follows
\begin{align*}
\log(1+\x) &= \log(\x(1+\x^{-1})) \\
& =   \log(\x) + \sum_{n=1}^{\infty} \frac{(-1)^{n+1}}{n}\x^{-n}\\
&= \log(\x) + \x^{-1} - \x^{-2}/2 + \ldots
\end{align*} 
Now, using $\x^\x = e^{\x \log(\x)}$ and $\exp(\eps) = \sum_{n\in \N} \eps^n/n! = 1 + \eps + \ldots$ for $\eps \prec 1$,  we obtain 
\begin{align*}
(\x+1)^\x &= \exp(\x \log(\x) +1 - \x^{-1}/2 + \ldots) \\
& = e \x^\x (1 -\x^{-1}/2 + \ldots) \\
& = e \x^\x - e2^{-1} \x^{\x-1} + \ldots
\end{align*}
Replacing $\x$ with $\omega$ we find the Ressayre form of the surreal $(\omega+1)^\omega$. 

\section{Finer asymptotic relations}\label{sec:equivalence}
The results in this section are stated for $\no$ but they hold more generally in every model of $T_{\exp}$.  We identify $\Sk$ as a subset of $\no$ as discussed in the previous section. In particular $\x = \omega \in \no$. 

\begin{defn}\label{defn:eqrel} Let $1 \leq c \in \no$. Given two positive surreals $f$ and $g$ we define $f\sim_c g$ if $f^c \sim g^c$ and $f\asymp_c g$ if $f^c \asymp g^c$. 
\end{defn}
When $c=1$, the relations $\sim_c$ and $\asymp_c$ become the usual $\sim$ and $\asymp$ relations. When $c>1$ we obtain finer equivalence relations. One of the main ideas of this paper is to try to understand how many classes modulo $\sim_c$ there are inside a class modulo $\asymp_c$. We are primarily interested in the case $c=1$, but we need to consider the general case to carry out the induction. In our terminology, the paper of \cite{Dries1984} deals with the case when $c$ is equal $\x^n$ for some $n\in \N$, but we need to follow a different approach to be able to generalize it. A consequence of our main result (\prettyref{thm:main}) is that the set of $\sim_c$-classes within any class modulo $\asymp_c$ has order type $\leq \omega$. 

In this section we establish some basic properties of $\sim_c$ and $\asymp_c$. In particular we show that $f \asymp_c g \iff c(f-g) \preceq g$ and $f \sim_c g  \iff c(f-g) \prec g$, yielding a characterization of these relations which does not depend on the exponential function.  
\begin{lem}
	For any $t\in \no$ we have $t\preceq 1$ if and only if $e^t \asymp 1$. 
\end{lem}
\begin{proof}
	We have $t\preceq 1$ if and only if there is $k\in \N$ such that $-k \leq t\leq k$. This happens if and only if $e^{-k} \leq e^{t} \leq e^k$ for some $k\in \N$, or equivalently $e^t \asymp 1$ (because $e^{-k}\asymp 1 \asymp e^k$).  
\end{proof}

\begin{lem}\label{lem:et1}
	For $t\in \no$ we have $t\prec 1$ if and only if $e^t \sim 1$. 
\end{lem}
\begin{proof}
	We have $t\prec 1$ if and only if $-1/k\leq t \leq 1/k$ for all positive $k\in \N$. This happens if and only if $e^{-1/k} \leq e^{t} \leq e^{1/k}$ for all positive $k\in \N$, or equivalently $e^t \sim 1$ (because $|e^t -1|\leq |e^{1/k} - e^{-1/k}|$ and $|e^{1/k} - e^{-1/k}|$ becomes smaller than any positive real for $k$ sufficiently large).  
\end{proof}


\begin{prop}\label{prop:mono}
	Let $c\geq c' \geq 1$ and let $f,g>0$. 
	\begin{enumerate}
		\item If $f\asymp_c g$, then $f\asymp_{c'} g$. 
		\item If $f\sim_c g$, then $f\sim_{c'} g$. 
	\end{enumerate}
	In particular, if $f\asymp_c g$, then $f\asymp g$. 
\end{prop}

\begin{proof}
	We first observe that, for $z\in \no^{>0}$ and $d\in \no^{\geq 1}$, we have $z<1 \implies z^d < z$ and $z>1 \implies z^d > z$, so in any case $|z-1| \leq |z^d-1|$. Taking $z = f/g$ and $d=c/c'$, we deduce that $|(f/g)^{c'}-1| \leq |(f/g)^c - 1|$. Thus if $f\sim_c g$, then $f\sim_{c'} g$. This proves (2). 
	
	To prove (1) assume that $f\asymp_c g$ and let $r\in \R^{>0}$ be such that $(f/g)^c \sim r$. Now observe that $(f/g)^{c'}$ is between $1$ and $(f/g)^c \sim r>0$, hence it is asymptotic to a positive real. 
\end{proof}

\begin{lem}\label{lem:asymp-c}
	For $c\ge1$ and $z>0$, we have
	\begin{enumerate}
		\item $z^c \asymp1 \iff z = 1 +
		O(1/c)$;
		\item $z^c\sim1 \iff z = 1 + o(1/c)$. 
	\end{enumerate}
\end{lem}
\begin{proof}
	The
case $c \preceq 1$ can be reduced the case $c=1$ using~\prettyref{prop:mono}.
If~$c\succ 1$, we can assume~$z\sim 1$, as otherwise both sides of either
equivalence are false.
We can thus write $z = 1+\eps$ for some $\eps \prec 1$. The results follow
from the following chains of equivalences.

\noindent\hbox to \textwidth{\hss%
\begin{minipage}{0.4\textwidth}\begin{align*} 
	& z^c \asymp 1\\
	\iff &e^{c\log(1+\eps)} \asymp 1 \\
	\iff & c\log(1+\eps) \preceq 1 \\
	\iff & \log(1+\eps)\preceq 1/c \\
	\iff & \eps \preceq 1/c
\end{align*}\end{minipage}\hss%
\begin{minipage}{0.4\textwidth}\begin{align*} 
	& z^c \sim 1\\
	\iff &e^{c\log(1+\eps)} \sim 1 \\
	\iff & c\log(1+\eps) \prec 1 \\
	\iff & \log(1+\eps)\prec 1/c \\
	\iff & \eps \prec 1/c
\end{align*}\end{minipage}\hss}
\vskip.5\baselineskip
\noindent where in the last step of both columns we used $\log(1+\eps)\sim \eps$ (which follows from $\eps\prec 1$).  
\end{proof}

\begin{prop} \label{prop:ftoc} 
	For $c\ge1$ and $f,g>0$, we have:
	\begin{enumerate}
		\item $f \asymp_c g \iff c(f-g) \preceq g$;
		\item $f \sim_c g  \iff c(f-g) \prec g$.
	\end{enumerate}
\end{prop}
\begin{proof}
By \prettyref{lem:asymp-c} with $z=f/g$. 
\end{proof}

\begin{prop} \label{prop:sim-c}Let $c>\N$, $z> 0$ and $r\in \R$. Then 
	$$z^c \sim e^r  \iff   z = 1 + r/c + o(1/c).$$
\end{prop}
\begin{proof} Let $z = 1 + r/c + o(1/c)$. Then $z^c=(1+r/c+o(1/c))^c \sim e^r$. Conversely, assume $z^c \sim e^r$. Then in particular $z^c \asymp 1$. By \prettyref{lem:asymp-c}, $z = 1 + O(1/c)$ so we can write $z = 1 + s/c + o(1/c)$ for some $s\in \R$. By the previous part $e^s \sim e^r$, hence $s= r$. 
\end{proof}

%

\begin{prop} Let $c\ge 1$ and $f,g,a,b>0$.
		\begin{enumerate}
		\item  if $f\asymp_c a$ and $g\asymp_c b$, then $fg\asymp_c ab$
and $f+g \asymp_c a+b$;
		\item If $f\sim_c a$ and $g\sim_c b$, then $fg \sim_c ab$
and $f+g\sim_c a+b$.
	\end{enumerate}
\end{prop}

\begin{proof} 
	Assume $f\asymp_c a$ and $g\asymp_c b$. By \prettyref{prop:ftoc}, $c(f-a) \preceq a$ and  $c(g-b) \preceq b$. Since $a,b$ are positive, $c(f-a) \preceq a+b$ and $c(g-b) \preceq a+b$. It follows that $c(f+g - (a+b)) \preceq a+b$, hence  $f+g\asymp_c a+b$.  
	
	In order to prove $fg\asymp_c ab$ we recall that $f\asymp_c a$
means $f^c \asymp a^c$ and $g\asymp_c b$ means $g^c \asymp b^c$. Multiplying we obtain 
	$(fg)^c \asymp (ab)^c$. 
 
 The  proof of second part is essentially the same: it suffices to replace $\preceq$ with $\prec$ and $\asymp_c$ with $\sim_c$. 
\end{proof}

\section{The support of a Skolem function}\label{sec:support}
We consider $\Sk$ as a substructure of $\no=\R((\M))_\on$ through the embedding induced by the identification $\x = \omega$. Given $f\in \Sk$, we can then write $f=\sum_{i<\alpha} \m_i r_i$ with $\alpha\in \on$, $\m_i\in \M$ and $r_i\in \R^*$. It thus makes sense to consider the {\bf support of a Skolem function}, that is, the set of monomials $\m_i$ which can appear in the above representation. We recall that a surreal number is an {\bf omnific integer} if it belongs to the subring $\no^\uparrow + \Z \subset \no$. We show that every Skolem function is an omnific integer, so it does not have infinitesimal monomials in its support. More generally we prove that a monomial in the support of a Skolem function is either $1$ or $\geq \x$ (so it cannot be $\log(\x)$ or $\sqrt{\x}$, say). To this aim we first show that every Skolem function belongs to a subfield $\K \subset \no$ which is similar to the field of transseries defined in \cite{DriesMM2001}, but unlike the transseries it is not closed under $\log$, although it is closed under $\exp$. 
\begin{defn}\label{defn:K}
	Let $\x = \omega\in \no$.  Working inside $\no$ we define
	\begin{enumerate}
		\item $\G_0 = \x^\Z$ and $\K_0 = \R((\G_0))$;
		\item $\G_{n+1} = e^{\K_n^\uparrow} \x^{\K_n^\uparrow+\Z} = e^{\K_n^\uparrow + \log(\x)(\K^\uparrow+\Z)}$ and $\K_{n+1} = \R((\G_{n+1}))$. 	
	\end{enumerate}Let $\G = \bigcup_n \G_n$ and let $\K = \bigcup_n \K_n \subseteq \R((\G))$. Finally, let $\K^\uparrow = \K\cap \no^\uparrow$. 
\end{defn}
We recall that a subfield of $\no$ is {\bf truncation closed} if whenever it contains $\sum_{i<\alpha} \m_i r_i$, it also contains its truncations $\sum_{i<\beta}\m_i r_i$ for all $\beta < \alpha$. Since $\K$ is an increasing union of the fields $\R((\G_n))$, it is obviously a subfield of $\no$ closed under truncations.
\begin{thm}\label{thm:K}
	$\K$ is a truncation closed subfield of $\no$ closed under $\exp$. If $f$ and $g$ are positive elements of the semiring $\K^\uparrow + \N \subset \K$, then $f^g = e^{g\log (f)} \in \K^\uparrow + \N$. It follows that $\Sk \subseteq \K^\uparrow + \N$. In particular every Skolem function is an omnific integer. 
\end{thm}

\begin{proof}
	For each $n\in \N$, $\G_n$ is a multiplicative group and therefore $\K_n$ is a field. Moreover $\G_0\subseteq \G_1$ and inductively $\G_n \subseteq \G_{n+1}$ and $\K_n \subseteq \K_{n+1}$. The fact that $\K$ is a truncation closed subfield of $\R((\G))$ is clear. To show that $\K$ is closed under $\exp$, let $x\in \K$ and write $e^x = e^{x^\uparrow} e^{x^\circ} e^{x^\downarrow}$. Now it suffices to observe that $e^{\x^\uparrow}\in \G$, $e^{\x^\circ}\in \R$ and $e^{\x^\downarrow} = \sum_{n\in \N} (\x^\downarrow)^n/n!\in \K$. More generally $\K$ is closed under the evaluation of a power series at an infinitesimal element. It remains to show that if $a,b$ are positive elements of $\K^\uparrow + \N$, then $a^b\in \K^\uparrow+\N$. 
	\begin{claim}\label{claim:mt}
		If $\m\in \G$ and $0<t\in \K^\uparrow$, then $\m^t\in \G$. 
	\end{claim}
	To prove the claim, write $\m = e^\gamma \x^{\theta + n}$ with $\gamma,\theta\in \K^\uparrow$ and $n\in \Z$. Then $\m^t  = e^{t \gamma} \x^{t(\theta + n)}\in \G$, as desired. 
	\begin{claim}
		Let $a$ and $b$ be positive elements of $\K^\uparrow + \N$. Then $a^b \in \K^\uparrow + \N$. Moreover if $a\geq 2$ (i.e. $a\neq 1$) and $b>\N$, then $a^b\in \K^\uparrow$. 
	\end{claim}
	We can write $b=b^\uparrow + n$ for some $n\in \N$ and $0< b^\uparrow \in \K^\uparrow$. Since $\K^\uparrow + \N$ is closed under finite products, $a^n\in \K^\uparrow + \N$. It remains to show that $a^{b^\uparrow}\in \K^\uparrow$. This is clear if $a\in \N$. If $a\nin \N$, we can write $$a = r\m (1+\eps)$$ where $1<\m \in \G$ is the leading monomial of $a$, $r\in \R^{>0}$ and $\eps\prec 1$. Then 
	$$a^{b^\uparrow} = r^{b^\uparrow}\m^{b^\uparrow} (1+\eps)^{b^\uparrow}.$$ 
	By Claim \ref{claim:mt} $\m^{b^\uparrow}\in \G$. By definition of $\G$ we also have $r^{b^\uparrow} = e^{{b^\uparrow}\log(r)} \in \G$. The third factor $(1+\eps)^{b^\uparrow}$ can be written in the form
	\begin{align*}
	(1+\eps)^{b^\uparrow} &= e^{{b^\uparrow}\log(1+\eps)}\\
	&={e^{({b^\uparrow}\log(1+\eps))}}^\uparrow {e^{(b^{\uparrow} \log(1+\eps))}}^\circ {e^{(b^{\uparrow} \log(1+\eps))}}^\downarrow.
	\end{align*}
	Since $\log(1+\eps) = \sum_{n=1}^\infty \frac{(-1)^{n+1}}{n}\eps^n \in \K$ and $b^\uparrow \in \K$, we have ${b^\uparrow}\log(1+\eps)\in \K$, so 
	$e^{({b^\uparrow}\log(1+\eps))^\uparrow}\in \G$. Moreover ${e^{(b^\uparrow  \log(1+\eps))}}^\circ \in \R$. The element $\delta = ({b^\uparrow}\log(1+\eps))^\downarrow$ is an infinitesimal element of $\K$ and  $e^{\delta}=\sum_{n\in \N}\frac{\delta^n}{n!}$ is a power series in $\delta$, so it belongs to $\K$. We have thus proved that $(1+\eps)^{b^\uparrow}\in \K$ and therefore $a^{b^\uparrow}\in \K$. 
	
	It remains to show that if $a\geq 2$, then $a^{b^\uparrow}$ is purely infinite. Since $a = r\m(1+\eps)$ is an omnific integer, each monomial in the support of $\eps$ is $\geq \m^{-1}$. It follows that  each monomial in the support of $(1+\eps)^{b^\uparrow}$ is $\m^{-n}$ for some $n\in \N$. Since $a^{b^\uparrow}=r^{b^\uparrow}\m^{b^\uparrow} (1+\eps)^{b^\uparrow}$, it follows that every monomial in the support of $a^{b^\uparrow}$ is $\geq r^{b^\uparrow}\m^{b^\uparrow}\m^{- n} = r^n(r\m)^{b^\uparrow-n}$, which is infinite. We conclude that 	$a^{b^\uparrow} \in \K^\uparrow$, as desired. 
	
	It follows from the claim that the set of positive elements of the semiring $\K^\uparrow + \N$ is closed under the operation $a,b\mapsto a^b$ and therefore it contains $\Sk$. 
	%
\end{proof}

\begin{prop}\label{prop:omnific}
	For every Skolem function $f$ there is a purely infinite surreal number $g$ and some $n\in \N$ such that $f = g+n$. Moreover $g$ is a Skolem function. 
\end{prop}
\begin{proof}
	By \prettyref{thm:K}, $f=g+n$ with $g\in \K^\uparrow$ and $n\in \N$, so we only need to show that $g\in \Sk$. 
	We proceed by induction on the formation of the Skolem terms. The case when $f$ is the sum or product of shorter terms is immediate. It remains to consider the case when $f = a^b$ with $a\geq 2$ and $b>\N$. In this case by \prettyref{thm:K}, $a^b\in \K^\uparrow$, so it is purely infinite. 
\end{proof}

\begin{thm}\label{thm:smallest-mon}
The monomial $\x$ is the smallest infinite monomial in $\K$. 
\end{thm}
\begin{proof}
	We prove by induction on $n\in \N$ that if $1<\m \in \M_n$, then $\m \geq \x$. This is clear for $n=0$ since $\M_0 = \x^\Z$. Let $1<\m \in \M_{n+1}$ and assume the result holds for the monomials in $\M_n$. By definition $\m= e^\gamma \x^{\theta +k}=e^{\gamma + \log(\x)(\theta+k)}$ with $\gamma,\theta \in \K_n^\uparrow$ and $k\in \Z$. By the induction hypothesis, $\x$ is the smallest infinite monomial in $\K_n$.
	If for a contradiction $1<\m<\x = e^{\log(\x)}$, then 
	$$0<\gamma + \log(\x)(\theta + k) < \log(\x).$$ 
	
	Case 1. If $\gamma \asymp \log(\x)(\theta + k)$, then $\log(\x) \asymp \frac{\gamma}{\theta+k}\in \K_n$, contradicting the induction hypothesis. 
	
	Case 2. If $\gamma \succ 	\log(\x)(\theta + k)$, then $0<\gamma <2 \log(\x)$, against the induction hypothesis. 
	
	Case 3. 
	If $\gamma \prec \log(\x)(\theta + k)$, we obtain $0 < \log(\x)(\theta + k) < 2\log(\x)$, whence $0 < \theta + k < 2$. Since $\theta$ is purely infinite and $k\in \Z$, we obtain $\theta=0$, hence $\gamma \prec \log(x)$, contradicting the induction hypothesis.  
\end{proof}

\begin{cor}\label{cor:support}
	If $\m$ is a monomial in the support of a Skolem function, then either $\m=1$ or $\m\geq \x$. 
\end{cor}
\begin{proof} Immediate from \prettyref{thm:smallest-mon} and the inclusion $\Sk\subset \K^\uparrow + \N$ (\prettyref{thm:K}).  
\end{proof}

\begin{cor}\label{cor:x-class} For $f,g\in \Sk$ we have:
	\begin{enumerate}
		\item $f\sim g \iff f = g + O(g/\x)$.
		\item $f^\x \asymp g^\x \iff f \sim g$.
	\end{enumerate}
\end{cor}
\begin{proof} The first part follows from \prettyref{thm:smallest-mon} and \prettyref{thm:K}, observing that $f-g\in \K$.  For (2) we take $z=f/g$ and $c=\x$ in \prettyref{lem:asymp-c} to obtain $f^\x \asymp g^\x$ if and only if $f =  g + O(g/\x)$. By  the first part this happens if and only if $f\sim g$. 
\end{proof}

\begin{cor}\label{cor:A-to-x}
	For any $A\subseteq \Sk$, we have $|A^\x/\! \asymp| = |A/\! \sim|$. 
\end{cor}
\begin{proof}
	By part (2) of  \prettyref{cor:x-class} . 
\end{proof}

\section{Components}
Let $f$ be a Skolem function. We say that $f$ is \textbf{additively irreducible} if it cannot be written as a sum of two smaller Skolem functions; $f$ is \textbf{multiplicatively irreducible} if it cannot be written as a product of two smaller Skolem functions; Following \cite{Levitz1978} we say that $f$ is a \textbf{component} if it is both additively and multiplicatively irreducible. 
\begin{rem}\label{rem:sumprodcomponents}
	We can write every Skolem function as a finite sum of finite products of components (not necessarily in a unique way). 
\end{rem} 
\begin{prop}\label{prop:normal-form}
	Every component has the form $1$, $\x$ or $f^g$. If $f^g$ is a component, then $f$ is multiplicatively irreducible and $g$ is additively irreducible. Every component $> \x$ can be written in the form $f^g$ where $f\geq 2$, $g\geq \x$, and $g$ is a component. 
\end{prop}
\begin{proof} A Skolem functions $<\x$ is a positive integers, so it is either $1$ or additively reducible. It follows that a component is either $1$, or $\x$, or $>\x$. In the latter case it must have the form $f^g$ (because it cannot be of the form $f+g$ or $fg$).
	The rest follows at once from the following identities: 
	\begin{itemize}
		\item 	if $f=f_1 f_2$, then $f^g = f_1^g f_2^g$;
		\item  if $g=g_1 + g_2$, then $f^g = f^{g_1}f^{g_2}$;
		\item  if $g=g_1g_2$, then $f^g = (f^{g_1})^{g_2}$. 
	\end{itemize}	 
\end{proof}
\begin{cor}\label{cor:exponents}
	For every Skolem function $h$ one of the following cases holds: 
	\begin{enumerate}
		\item $h = f \cdot g$ where $f \geq \x$ and $g\geq \x$;
		\item $h = f^{g}$ where $f \geq 2$ and $g$ is a component $\geq \x$;
		\item $h = f + g$, where $f \asymp h$ and $f$ is a component;
		\item $h=1$ or $h=\x$. 
	\end{enumerate}
\end{cor}

\section{Main theorem}
We work inside the surreal numbers $\no$ and identify $\Sk$ as a subset of $\no$, with $\x=\omega\in \no$. Our main result is the following. 

\begin{thm} \label{thm:main}  
Let $c\geq 1$ be a surreal number and let $Q\in \Sk$.  
The set of real numbers $r\in \R^{>0}$ such that there is $h\in \Sk$ satisfying $(h/Q)^c \sim r$, is well ordered and has no accumulation points in $\R$ (hence it has order type $\leq \omega$). 
\end{thm}

\begin{proof}
Given $Q$ and $c$, the set of reals $r\in \R^{>0}$ such that there is $h\in \Sk$ with $(h/Q)^c \sim r$ is in order preserving bijection with the set of Skolem functions $\asymp_c Q$, so it is well ordered (as $\Sk$ is well ordered). A well ordered subset of $\R^{>0}$ has an accumulation point if and only if it contains a strictly increasing 
and bounded sequence. Assuming for the sake of a contradiction that the theorem fails, let $Q$ be minimal in the well order of~$\Sk$ such that
there exist a surreal number $c\geq 1$, a strictly increasing and bounded sequence of positive real
numbers~$(r_n)_{n\in\N}$, and a
sequence~$(h_n)_{n\in\N}$ of Skolem functions, such that
$$(h_n/Q)^c \sim r_n.$$ 
By the assumptions, for all $n\in \N$, we have 
\begin{equation*}
h_n \asymp_c Q
\end{equation*}
which in turn implies
$h_n \asymp Q$
(by~\prettyref{prop:mono}).
In other words, all the functions $h_n$ belong to the
archimedean class of~$Q$. Let us also notice that, given $c\geq 1$ as above, the minimality property of $Q$ implies that $Q$ is minimal in its $\asymp_c$-class in $\Sk$ (using the fact that if $Q'\asymp_c Q$, there is $s\in \R^{>0}$ with $(h_n/Q')^c \sim sr_n$ for all $n\in \N$).  

Now let $h$ be the least Skolem function $\asymp Q$,
and note that its multiples $nh$ ($n\in \N$) are cofinal in the
archimedean class of $Q$. There is $N\in \N$ such that $h_n \leq Nh$ for all $n\in \N$, for otherwise the sequence $(r_n)_{n\in \N}$ is unbounded.
We call the least such $N$ the {\bf characteristic bound} of the sequence
$(h_n)_{n\in \N}$. 

We now choose $(h_n)_n$ with the additional property that $(h_n)_n$ has minimal characteristic bound $N\in \N$. Note that the characteristic bound $N$ is only defined for those sequences $(h_n)_n$ such that there is $c\geq 1$ and a strictly increasing bounded sequence $r_n\sim (h_n/Q)^c$ as above, but it does not depend on the choice of $c$, so we can minimize $N$ before choosing $c$. Finally we fix the exponent $c\geq 1$ and we get a strictly increasing bounded sequence $(r_n)_n$ of positive real numbers such that $(h_n/Q)^c \sim r_n$. 
	
Along the sequence $(h_n)_n$ there is one of the cases of \prettyref{cor:exponents} which holds infinitely often. By taking a subsequence we can thus assume to be in one of the following cases:
\begin{enumerate}
\item for all $n\in\N$, $h_n = f_n \cdot g_n$ where $f_n \geq \x$ and $g_n\geq \x$;
\item for all $n\in\N$, $h_n = f_n^{g_n}$ where $f_n \geq 2$ and $g_n \geq \x$;
\item for all $n\in\N$, $h_n = f_n + g_n$, where $f_n \asymp Q$ and $f_n$ is a component;
\end{enumerate}
with $(f_n)_n$ and $(g_n)_n$ weakly increasing
(taking advantage of the fact that $\Sk$ is well ordered).

We will need the following observation. Define $r\in \R^{>0}$ by $(Q/h_0)^c\sim r$  and let $r'_n = r_nr$. Notice that $$(h_n/h_0)^c \sim
r'_n$$ for all $n\in\N$ and observe that
$(r'_n)_n$ is again increasing and bounded.
\smallskip

\textbf{Case 1}. Suppose $h_n = f_n \cdot g_n$ where $f_n\geq \x$ and $g_n\geq \x$ for all $n\in \N$. By our
assumptions $r'_n \sim (h_n/h_0)^c = (f_n/f_0)^c (g_n/g_0)^c $. Both
factors in the last expression are  $\geq 1$ because the sequences
$(f_n)_n$ and $(g_n)_n$ are weakly increasing. It then follows that there
are real numbers $s_n\geq 1$ and $t_n\geq 1$ such that 
\begin{equation*}
(f_n/f_0)^c \sim
s_n, \quad \quad \quad (g_n /g_0)^c \sim t_n
\end{equation*}
 and $r'_n = s_n t_n$. Since $(r'_n)_n$ is
bounded, the sequences $(s_n)_{n\in \N}$ and $(t_n)_{n\in \N}$ must also be
bounded. Since both $f_n$ and $g_n$ are $\geq \x$ and their product is
$h_n$, they are both $\prec h_n \asymp Q$. In particular $f_0$ and $g_0$ are $\prec Q$. By the minimality of~$Q$, the
sequences $(s_n)_n$ and $(t_n)_n$ are eventually constant, hence
$(r'_n)_n$ is eventually constant, a contradiction.

\smallskip
	In the next case we use the full strength of the fact that we work with all the equivalence relations $\sim_c$ and not only with $\sim$. 
\smallskip	
	
	\textbf{Case 2.} Suppose $h_n=f_n^{g_n}$ where $f_n\geq 2$ and $g_n\geq \x$ for all $n\in \N$. Note that $r'_n \sim (h_n/h_0)^c \geq h_n/h_0 = f_n^{g_n}/f_0^{g_0} \geq f_0^{g_n-g_0} \geq 2^{g_n-g_0}$ for $n\in \N$. Since $(r'_n)_n$ is bounded in $\R$,  there is $M \in \N$ such that $g_n-g_0 < M$ for all $n\in \N$. If the difference between two Skolem functions is bounded by a natural number, then it is equal to a natural number (\prettyref{prop:omnific}). Since $(g_n)_{n\in \N}$ is weakly increasing, there must be some $k\in \N$ such that $g_n = g_k$ for all $n\ge k$. For $n\geq k$ we have $(h_n/h_k)^c \sim s r'_n$ where $s\in \R^{>0}$ is defined by $s \sim  (h_0/h_k)^c$. 
		 Taking a subsequence we can assume $k=0$. Thus $s=1$ and
		$$r'_n \sim (h_n/h_0)^c = (f_n/f_0)^{g_0 c}$$
		for all $n\in \N$. Since $f_n\geq 2$ and $g_n\geq \x$, we have $f_n \prec f_n^{g_n} = h_n \asymp Q$ for all $n\in \N$. Since $(f_n/f_0)^{g_0 c} \sim r'_n$ and $f_0 \prec Q$, by the minimality of $Q$ we deduce that $(r'_n)_{n\in \N}$ is eventually constant, a contradiction.

		\smallskip
		 We have shown that a sequence $(h_n)_n$ with minimal characteristic bound falls necessarily under case 3, so it cannot consist entirely of components. It remains to deal with case 3. 
		
		\smallskip
	\textbf{Case 3.} Suppose that $h_n = f_n + g_n$  where $f_n \asymp Q$ and $f_n$ is a component for all $n\in \N$.  
It suffices to consider the cases $c=1$ and~$c>\N$, for if $c\asymp c'$ and $(h_n/Q)^c \sim r_n$, then $(h_n/Q)^{c'} \sim r_n^t$, where $t\in \R^{>0}$ is such that $t \sim c'/c$. 
	Taking a subsequence we can further assume that either $g_n \asymp Q$ for all $n\in \N$, or $g_n \prec Q$ for all $n\in \N$.  
	
%
%
\smallskip

		\textbf{Case $c=1$}. The assumption $(h_n/Q)^c \sim r_n$ becomes $h_n/Q \sim r_n$. Recall that $h_n = f_n+g_n$. Consider first the subcase with $g_n \asymp Q$ for all $n\in \N$. Then all the functions $h_n,f_n,g_n$ are in the archimedean class of $Q$, so there are positive real numbers $a_n\in \R^{>0}$ and $b_n\in \R^{>0}$ such that 
		$$a_n \sim f_n/Q \quad  \text{ and } \quad b_n \sim g_n/Q$$
		 for all $n\in \N$. It follows that $a_n+b_n = r_n$ for
all $n\in \N$. Since $(r_n)_n$ is bounded, it follows that $(a_n)_n$ and
$(b_n)_n$ are also bounded. Recall that $Q$ is minimal in its
\hbox {$\asymp_c$-class}. Since $c=1$, this means that $Q$ is minimal in its archimedean class, so the functions $h_n,f_n,g_n$ are all $\geq Q$. If $N$ is the characteristic bound of $(h_n)_n$, we have $f_n \geq Q$ and $g_n \geq Q$ and $f_n+g_n = h_n \leq NQ$, so both $(h_n)_n$ and $(g_n)_n$ have characteristic bound $\leq N-1$. By the minimality of $N$, we deduce that the sequences $(a_n)_n$ and $(b_n)_n$ are eventually constant, hence their sum $(r_n)_n$ is also eventually constant, a contradiction. 
		 
		 \smallskip
		 Now consider the subcase with $g_n \prec Q$ for all $n\in \N$. Then the functions $h_n$ and $f_n$ are in the archimedean class of $Q$, but $g_n$ is in a lower archimedean class. It follows that 
		$$r_n \sim h_n/Q = (f_n+g_n)/Q \sim f_n/Q$$
		 for all $n\in \N$. 
		The sequence $(f_n)_n$ is then a counterexample to the theorem with the same characteristic bound than  $(h_n)_n$ but consisting entirely of components. We have already shown that this cannot happen, so we have a contradiction.  
		 \smallskip
		
		\textbf{Case $c>\N$.} 
		We are still inside the case $h_n = f_n+g_n$ with $h_n$ a component.  By \prettyref{prop:sim-c} the condition $(h_n/h_0)^c \sim r'_n$ can be rewritten in the form 
		\begin{equation}\label{f+g}
		h_n/h_0 - 1 = s_n/c + o(1/c)
		\end{equation}
		where $s_n = \log(r'_n)$ for $n\in \N$. Note that since $(h_n)_n$ is increasing, we have $r'_n \geq 1$, so $\log(r'_n)$ is well defined and $\geq 0$. Moreover $(s_n)_n$ is strictly increasing. Using $h_n = f_n+g_n$, Equation (\ref{f+g}) becomes
		$$(f_n-f_0) + (g_n-g_0) = s_n(f_0+g_0)(1/c +o(1/c)).$$ 
		Dividing by $f_0$ and multiplying by $c$, it can be rewritten as 
		\begin{equation*}
		c\left(\frac{f_n}{f_0}-1\right) + 
		\left(\frac{c}{f_0} \right) 
		\left( g_n - g_0\right) 
		=
		s_n \left( 1 + \frac{g_0}{f_0} \right) + o(1)
		\end{equation*}
		Since $g_0 \preceq Q \asymp f_0$ the right-hand-side is
finite. The two summands on the left are $\geq 0$ and their sum is finite,
so they are both finite, i.e.\ they can be written as a real number plus an infinitesimal. This means that we can define $a_n\in \R$ and $b_n\in \R$ by the equations
		\begin{equation*}
		a_n = c\left(\frac{f_n}{f_0}-1\right) + o(1) \quad \text{and} \quad b_n = \left(\frac{c}{f_0} \right) 
		\left( g_n - g_0\right)  + o(1).
		\end{equation*}
		We can then write
		\begin{equation}\label{sn}
		a_n + b_n = s_n \left( 1 + \frac{g_0}{f_0}\right) + o(1).
		\end{equation}
		Since $(f_n)_n$ and $(g_n)_n$ are weakly increasing, the sequences of real numbers $(a_n)_n$ and $(b_n)_n$ are weakly increasing. Moreover, since $(s_n)_{n\in \N}$ is bounded and $g_0/f_0$ does not depend on $n$, $(a_n)_{n\in \N}$ and $(b_n)_{n\in \N}$ are also bounded. 
		
		By \prettyref{prop:sim-c} (and the assumption $c>\N$) the definition of $a_n$ can be rewritten in the form 
		$$(f_n/f_0)^c \sim e^{a_n}.$$
		We claim that $(a_n)_n$ is eventually constant. If $f_0<Q$ this follows from the minimality property of $Q$, so we can assume $Q\leq f_0$. We also have $f_0 \leq h_0 \leq h_n \asymp_c Q$, so all the functions $f_n$ are in the $\asymp_c$-class of $Q$ and therefore there is a real number $s\in \R^{>0}$ such that 
		$$(f_n/Q)^c \sim se^{a_n}$$
		for all $n\in \N$. Assuming for a contradiction that $(a_n)_n$ is not eventually constant, $(f_n)_n$ would be a counterexample to the theorem with a characteristic bound at most equal to that of $(h_n)_n$ (because $f_n \leq h_n$). However $(f_n)_n$ has the additional property that it consists entirely of components and we have already shown that this cannot happen. This contradiction shows that $(a_n)_n$ is indeed eventually constant. 
		
		\medskip
		We now claim that $(b_n)_{n\in \N}$ is eventually constant.   By our definitions we have $b_n = (c/f_0)(g_n - g_0)+o(1)$ so we can write 
		\begin{equation}\label{eq:bn}
		g_n - g_0 = b_n P + o(P)
		\end{equation}
	where $P= f_0/c$. 

	We distinguish three subcases. 
		\smallskip
		
		Subcase 1. If $g_0 \prec P$, then for all $n$ we have $g_n = b_n P+o(1)$, or equivalently $g_n/P = b_n+o(1)$. Since $P=f_0/c \prec f_0 \asymp Q$, by the minimality of $Q$ we conclude that $(b_n)_{n\in \N}$ is eventually constant (possibly $0$), as desired. 
		
		Subcase 2. If $g_0 \asymp P$, then there is $r\in \R^{>0}$ such that $g_0 \sim r P$, so $g_n = (b_n+r) P+o(1)$ for all $n\in \N$. Reasoning as above,  $(b_n+r)_{n\in \N}$ is eventually constant, hence so is $(b_n)_{n\in \N}$. 
		
Subcase 3.
If $g_0 \succ P$, we divide Equation (\ref{eq:bn}) by $g_0$ obtaining
$(g_n/g_0 - 1) = b_n (P/g_0) + o(P/g_0)$.
Now we multiply by $c'= g_0/P$ to get $c' (g_n/g_0-1) = b_n + o(1)$. 
Since $c' >\N$, by \prettyref{prop:sim-c} we obtain $\left(\frac{g_n}{g_0}\right)^{c'}\sim e^{b_n}$.
If $g_0<Q$, then by the minimality of $Q$ we conclude that $(b_n)_{n\in \N}$ is eventually constant, as desired.
In the opposite case we have $g_0 \asymp_c Q$ (since $Q\le g_0 \le h_0
\asymp_c Q$).
The new exponent $c'$ is in the same archimedean class of~$c$, because
$c'=(g_0/f_0)c$ and $g_0\asymp f_0$, thus $c/c' = \gamma + o(1)$ for some
real~$\gamma>0$. Raising to the power $\gamma$ both sides of the relation  $\left(\frac{g_n}{g_0}\right)^{c'}\sim e^{b_n}$ we then obtain $\left(\frac{g_n}{g_0}\right)^{c} \sim e^{b_n\gamma}$.
Since $g_0 \asymp_c Q$, there is a real~$r>0$ such that
$\left(\frac{g_n}{Q}\right)^{c} \sim r e^{b_n\gamma}$.
All the functions $h_n$, $f_n$, $g_n$ are in the same archimedean class,
namely that of~$Q$. Since $h_n = f_n + g_n$, it follows that the characteristic bound
of~$(g_n)_n$ is lower than the characteristic bound~$N$ of~$(h_n)_n$. By
the minimality of $N$, we deduce that $(re^{b_n\gamma})_n$ is eventually
constant, hence also $(b_n)_n$ is eventually constant, as desired. 
		
		From Equation (\ref{sn}) we can now conclude that $(s_n)_{n\in \N}$ is also eventually constant, against the assumptions. This contradiction concludes the proof. 
	\end{proof}
\begin{cor}\label{cor:classes}
	Let $1\leq c \in \no$. The set of $\sim_c$-classes of Skolem functions within any class modulo $\asymp_c$ has order type $\leq \omega$. In particular, the set of asymptotic classes of Skolem functions within any archimedean class has order type $\leq \omega$. 
\end{cor}
\begin{proof}
	Fix $Q\in \Sk$. For every $h\asymp_c Q$, the $\sim_c$-class of $h$ is determined by the real number $r\in \R^{>0}$ defined by $r\sim (h/Q)^c$, so we can apply \prettyref{thm:main}.  
\end{proof}
We need the following corollary to obtain an upper bound on the order type of the set of Skolem functions $<2^{\x^\x}$.

\begin{cor}\label{cor:classes-within} For any $A\subseteq \Sk$, $|A^\x /\! \asymp |   \;= \;   |A/\! \sim \! | \;\leq\;\omega |A/\! \asymp \!| $. 
\end{cor}
\begin{proof}
The first equality is \prettyref{cor:A-to-x}. The inequality $ |A/\! \sim \!| \;\leq\;\omega |A/\! \asymp \!| $ follows from \prettyref{cor:classes}. 
\end{proof}

We give below other consequences of the main theorem. 
\begin{cor}\label{cor:discreteness}
	Let $Q= \sum_{i<\alpha} \m_i r_i \in \no$ and let $\m$ a monomial smaller than all monomials $\m_i$ in the support of $Q$. Then there is a well ordered subset $D\subseteq \R$ without accumulation points such that for every Skolem function $f$, if $f$ (seen as an element of $\no$) has a truncation of the form $Q+r\m$, then $r\in D$. 
\end{cor}

\begin{proof} If $Q = 0$ the desired result is an immediate consequence of \prettyref{thm:main}. Assume $Q \neq 0$. 
	We can write $f = Q + r\m + o(\m)$. Thus $f/Q = 1 + r \m/Q + o(\m/Q)$. Let $c = Q/\m$. Then $c>\N$ and $f/Q = 1 +  r/c + o(1/c)$. By \prettyref{prop:sim-c}, $(f/Q)^c \sim e^r$. Let $D$ be the set of possible values of $e^r$ as $f$ varies. Since $\Sk$ is well ordered, $D$ is well ordered. Suppose for a contradiction that there is an increasing sequence $e^{r_n}\in D$ with an accumulation point $e^r\in \R$. We can then find $f_n\in \Sk$ with $(f_n/Q)^c \sim e^{r_n}$, contradicting \prettyref{thm:main}.   
\end{proof}

By \prettyref{thm:smallest-mon}, given two Skolem functions $f,g$, the smallest infinite monomial in the support of $f/g$ (seen as a surreal number) is $\x= \omega$. We thus obtain the following result, which extends to the whole class $\Sk$ the corresponding result of \citet{Dries1984} for the fragment below $2^{2^\x}$. 

\begin{cor}\label{cor:H} Let $g\in \Sk$. 
	For every finite sequence $r_0,\ldots, r_k$ of real numbers (empty if $k=-1$), there is a well ordered subset $R = R(g,r_0,\ldots, r_k) \subseteq \R$ without accumulation points such that for every $f\asymp g$ in $\Sk$ satisfying $$f/g = r_0 + r_1/\x + \ldots + r_k/\x^k + r_{k+1}/\x^{k+1}+ O(1/\x^{k+2})$$ we have $r_{k+1}\in R$. 
\end{cor}

\section{Levitz's regular functions}
We say that $f\in \Sk$ is an \textbf{additive scale} if the sum of two Skolem functions less than $f$ is less then $f$. We define $f$ to be a \textbf{multiplicative scale} if the product of two Skolem functions less then $f$ is less then $f$. 
Clearly every additive scale is additively irreducible and every multiplicative scale is multiplicatively irreducible. It is also easy to see that a multiplicative scale $f\neq 2$ is an additive scale. Indeed if $f$ is not an additive scale, there is $g<f$ with $g+g \geq f$. Since $f\neq 2$, we have $g\neq 1$, so $gg \geq g+g \geq f$, contradicting the fact that $f$ is a multiplicative scale. We have thus proved that a multiplicative scale $\neq 2$ is a component (recall that $f$ is a component if it is both additively and multiplicatively irreducible). 
%

We say that $h\in \Sk$ is \textbf{regular} if $h\neq 1$ and for all Skolem functions $f<h$ we have $f^\x < h$. Regular functions play a crucial role in the work of  \citet{Levitz1978}. Every regular function is a multiplicative scale, so it is either equal to $2$ or a component. In the rest of the sections we characterize the regular functions $\leq 2^{\x^\x}$. 

\begin{prop}\label{prop:small-components}
	The components $<\x^\x$ are $1,\x$ and $p^\x$ with $p\in \N$ prime. 
\end{prop}
\begin{proof}
	If $h$ is a component $>\x$, we can write $h = f^g$ where $f\geq 2$ is multiplicatively irreducible and $g$ is a component $\geq \x$ (\prettyref{prop:normal-form}). Since $h  <\x^\x$, we must have $g\leq \x$ and $f<\x$, so $h = p^\x$ with $p$ a prime in $\N$. 
\end{proof}

\begin{lem} \label{lem:stratification} Let $n>0$. 
	If $f$ is a Skolem function $<2^{(n+1)^\x}$, then there is $k\in \N$ such that $f<2^{n^\x \x^k}$. It follows that $f^\x < 2^{(n+1)^\x}$, hence $2^{(n+1)^\x}$ is a regular function. 
\end{lem}

\begin{proof} For a contradiction let $n>0$ be minimal such that the statement fails. Let $H_n\subset \Sk$ be the set of Skolem functions bounded by one of the functions $2^{n^\x \x^k}$ as $k$ ranges in $\N$.  Let $f< 2^{(n+1)^\x}$ be the minimal Skolem function such that $f\nin H_n$. We need to consider the following cases. 
	\begin{enumerate}
		\item $f$ is not a component;
		\item $f$ is either $1$ or $\x$;
		\item $f$ is a component of the form $a^\x$; 
		\item $f$ is a component of the form $a^b$ with $b>\x$. 
	\end{enumerate}
	Case (2) is clearly impossible. Cases (1) and (3) are also impossible by the minimality of $f$ and the fact that $H_n$ is closed under sums, products and exponentiation to the power $\x$. Finally, in case (4),  by \prettyref{prop:small-components}, we can write $b= p^\x$ with $p$ prime $<n+1$. Let $q\in \N$ be mimimal such that $qp \geq n+1$ and notice that $2 \leq q\leq n$. We must have $a < (2^{q^\x})$, so by the minimality of $n$ we have  $a<2^{(q-1)^\x \x^k}$ for some $k\in \N$. It follows that $f=a^b < 2^{((q-1)p)^\x \x^k}\leq 2^{n^\x \x^k}$. 
\end{proof}

\begin{prop}\label{prop:sup2nx}
	Let $f<2^{\x^\x}$ be a Skolem function. Then there is $n\in \N$ such that $f<2^{n^\x}$. It follows that $2^{\x^\x}$ is the smallest regular function bigger than $2^{n^\x}$ for all positive $n\in \N$. 
\end{prop}

\begin{proof} Let $f<2^{\x^\x}$ and assume by induction that the proposition holds for all Skolem functions $<f$. If $f$ is not a component, then it is a sum of products of smaller functions, and we conclude observing the 
	the Skolem functions less that $2^{n^\x}$ for some $n\in \N$ form an initial segment closed under sums and products. The cases $f=1$ or $f=\x$ are trivial. It remains to consider the case when $f$ is a component of the form $a^b$ where $b>1$ is a component and $a\geq 2$. Since $f=a^b <2^{\x^\x}$, we have $b<\x^\x$. By \prettyref{prop:small-components}, either $b=\x$ or $b = p^\x$ for some prime $p\in \N$.  By the induction hypothesis $a<2^{m^\x}$ for some $m\in \N$, hence in the first case $f=a^b < 2^{{(m+1)}^\x}$ and in the second case $f=a^b < 2^{m^\x p^\x}= 2^{(mp)^\x}$. In either case $f\leq 2^{n^\x}$ for a suitable $n$. 
\end{proof}

\section{The fragment of van den Dries and Levitz}
Van den Dries and Levitz \cite{Dries1984} proved that $|2^{2^\x}| = \omega_3 = \omega^{\omega^\omega}$. As a preparation for the results in the next section we give a proof of the inequality $|2^{2^\x}| \leq \omega_3$ based on \prettyref{cor:classes-within}.  Thanks to the fact that \prettyref{cor:classes-within} holds for the whole class $\Sk$, we shall then be able to extend the result to bigger fragments with a similar technique. 

We recall that given a set $X\subseteq \Sk$, $\sum X$ is the set of finite non empty sums of elements of $X$ (we exclude the empty sum because $0$ is not a Skolem function). Similarly, we write $\prod X$ for the set of finite products of elements of $X$, with the convention that the empty product is $1$.

\begin{thm}[\citep{Dries1984}] \label{thm:DL}  $|2^{2^\x}|\leq \omega^{\omega^\omega}$. Moreover the set of archimedean classes of the set of Skolem functions $<2^{2^\x}$ has order type $\leq \omega^\omega$. \end{thm}
\begin{proof} 
	Let $A$ be the set of Skolem functions $< 2^{2^\x}$. We need to prove that 
	$|A|\leq \omega_3$ and $\class{A} \leq \omega_2$, where $A/\! \asymp$ is the set of $\asymp$-classes of elements of $A$.  
	
	By \prettyref{lem:stratification} for $n=1$ we can write 
	$$A = \bigcup_{d\in \N} S_{d}$$ 
	where $S_d$ is set of Skolem functions $<2^{\x^d}$. 
	
	By induction on $d$ we show that   
	$$|S_{d}|<\omega_{3} \quad \quad  \text{and} \quad \quad \class{S_{d}}<\omega_{2}.$$
	Granted this, the supremum over $d$ of these ordinals is $\leq \omega_{3}$ and $\leq \omega_{2}$ respectively, yielding the desired bounds $|A|\leq \omega_3$ and $\class{A} \leq \omega_2$.  
	
	The case $d=0$ of the inductive proof is obvious, so assume $d>0$. Writing a Skolem function as a finite sum of finite products of components, and observing that $g^\x \leq 2^{\x^d} \implies g < 2^{\x^{d-1}}$, we obtain
	$$S_{d} \subseteq \sum \prod (\x^\N \cup S_{d-1}^\x )$$ 
	By the induction hypothesis $|S_{d-1}| < \omega_{3}$ and $\class{S_{d-1}}<\omega_{2}$. Now observe that $|S_{d-1}^\x| = |S_{d-1}| < \omega_3$. Moreover by \prettyref{cor:classes-within} we have
	$$\class{S_{d-1}^\x} \;\leq \; \omega\class{S_{d-1}} \;<\;\omega_{2}$$
	(because the set of ordinals $<\omega_2$ is closed under multiplication by $\omega$). Letting $X = \x^\N\cup S_{d-1}^\x$, it follows that $|X|<\omega_3$ and $\class{X}<\omega_2$. Now observe that each element of $\prod X$ is a product of at most $2$ elements of $X$ (because $\x^\N$ and $S_{d-1}^\x$ are closed under finite products). By \prettyref{cor:binary-bounds} we then obtain $|\prod X|< \omega_{3}$ and $\class{\prod X}<\omega_2$. By  \prettyref{cor:bounds-on-sums} we conclude that $|\sum \prod X| < \omega_{3}$ and $\class{\sum \prod X} < \omega_{2}$. Since $S_{d}$ is included in $\sum \prod X$ we get the desired bounds.
\end{proof}

\section{Fragments bounded by larger regular functions}
We have seen that $|2^{2^\x}| \leq \omega_3$. The following result gives bounds on $|2^{n^\x}|$. In particular $|2^{3^\x}| \leq \omega_4$, $|2^{4^\x}| \leq \omega_5$, and so on. 

\begin{thm}\label{thm:2nx}
	Let $1\leq n\in \N$. Then $|2^{(n+1)^\x}| < \omega_{n+2}$. 
	Moreover the set of archimedean classes of the set of Skolem functions $<2^{(n+1)^\x}$ has order type $\leq \omega_{n+1}$.
\end{thm}
\begin{proof}  Let $A_n$ be the set of all Skolem functions $<2^{(n+1)^\x}$. We prove by induction on $n$ that $|A_n| \leq \omega_{n+2}$ and $\class{A_n}\leq \omega_{n+1}$.   
	
	For $n=1$, $A_n$ is the set of Skolem functions $<2^{2^\x}$ so we can apply \prettyref{thm:DL}. 
	
	Assume $n>1$.  For $d\in \N$, let $S_{n,d}$ be the set of Skolem functions $<2^{n^\x \x^d}$. By \prettyref{lem:stratification}  $$A_n = \bigcup_{d\in \N} S_{n,d}.$$ By a secondary induction on $d$ we show that   
	$$|S_{n,d}|<\omega_{n+2} \quad \quad  \text{and} \quad \quad \class{S_{n,d}}<\omega_{n+1}.$$
	Granted this, the sup over $d$ of these ordinals is $\leq \omega_{n+2}$ and $\leq \omega_{n+1}$ respectively, yielding the desired bounds 
	$$|A_n|\leq \omega_{n+2} \quad \quad \text{and} \quad \quad \class{A_n} \leq \omega_{n+1}.$$  
	The case $d=0$ of the secondary induction follows from  $$S_{n,0} = A_{n-1}$$ applying the primary induction on $n$. 
	
	Assume $d>0$. We claim that
	$$S_{n,d} \subseteq \sum \prod (\x^\N \cup S_{n,d-1}^\x \cup A_{n-2}^{2^\x} \cup \ldots \cup A_{n-2}^{n^\x}).$$ 
	To prove the claim, it suffices to show that if $h\in S_{n,d}$ is a component, then it belongs to  
	$\x^\N \cup S_{n,d-1}^\x \cup A_{n-2}^{2^\x} \cup \ldots \cup A_{n-2}^{n^\x}$. We can assume that $h>\x$, so we can write $h$ in the form $h=f^g$ where $f\geq 2$ and $g$ is a component $\geq \x$ (\prettyref{prop:normal-form}). Since $h<2^{\x^\x}$, we have $g<\x^\x$, so either $g= \x$ or $g = p^\x$ for some prime $p\in \N$ (\prettyref{prop:small-components}).  Since $h\in S_{n,d}$, we have $h= f^g<2^{n^\x \x^d}$. So if $g=\x$ we get $f < 2^{n^\x \x^{d-1}}$ and therefore $h \in S_{n,d-1}^\x$. On the other hand if $g = p^\x$, then from $f^g<2^{n^\x \x^d}< 2^{(n+1)^\x}$ we obtain $f< 2^{(n-1)^\x}$ and $p\leq n$, so $h\in A_{n-2}^{2^\x} \cup \ldots \cup A_{n-2}^{n^\x}$ and the claim is proved. 
	
	By the primary induction $|A_{n-2}| \leq \omega_{n}$. By the secondary induction $|S_{n,d-1}| < \omega_{n+2}$ and $\class{S_{n,d-1}}<\omega_{n+1}$. It follows that $|S_{n,d-1}^\x| = |S_{n,d-1}| < \omega_{n+2}$. Moreover by \prettyref{cor:classes-within} we have
	$$\class{S_{n,d-1}^\x} \; \leq \; \omega  \class{S_{n,d-1}} <\omega_{n+1}.$$ 
	We also have $|A_{n-2}^{k^\x}| = |A_{n-2}|$ and $\class{A_{n-2}^{k^\x}} \leq |A_{n-2}^{k^\x}|$. Taking the union of these sets it follows that the set $X=\x^\N \cup S_{n,d-1}^\x \cup A_{n-2}^{2^\x} \cup \ldots \cup A_{n-2}^{n^\x}$ satisfies $|X|<\omega_{n+2}$ and $\class{X}< \omega_{n+1}$. The same bounds hold for $\prod X$ because each element of $X$ is a product of at most $n+1$ elements of $X$ (as $X$ is the union of $n+1$ sets closed under products). By  \prettyref{cor:bounds-on-sums} we conclude that $|\sum \prod X| < \omega_{n+2}$ and $\class{\sum \prod X} < \omega_{n+1}$. Since $S_{n,d}$ is included in $\sum \prod X$ we get the desired bounds.
\end{proof}

\begin{thm}\label{thm:2xx}
	The set of Skolem functions $<2^{\x^\x}$ has order type $\leq \eps_{0}$. 
\end{thm}
\begin{proof} 
	Immediate from \prettyref{thm:2nx} and \prettyref{prop:sup2nx}. 
\end{proof}

\section{Exponential constants}
Let $\E^+\subseteq \R^{>0}$ be the smallest set of real numbers containing $1$ and closed under $+,\cdot, ^{-1}$ and $\exp$. Let $\E = \E^+ - \E^+$. Note that $\E$ is a subring of $\R$, $\exp(\E)\subseteq \E^+$ and $\E^+\subseteq \E$ (because $1\in \E$ and $\E^+\cdot \E \subseteq \E$). The following result is inspired by the  final remarks of \cite{Dries1984}. The authors gave a detailed proof for the fragment below $2^{2^\x}$, working with Laurent expansions rather than Ressayre forms, and announced a proof for the whole class $\Sk$ using the embrionic form of the transseries in \cite{Dahn1984a}.

\begin{prop} Let $f = \sum_{i<\gamma}e^{\gamma_i} c_i\in \no$ be the Ressayre form of a Skolem function $f$. Then $c_0\in \E^+$ and $c_i \in \E$ for every $i<\alpha$. 
\end{prop}
\begin{proof} By induction on the formation of $f$. The cases $f=a+b$ or $f=a\cdot b$ are straightforward, so it suffices to consider the case $f = a^b$ with $a\geq 2$ and  $b>\N$ (note that in this case $b$ is purely infinite). By definition 
	$$a^b = {e^{(b\log(a))}}^\uparrow {e^{(b \log(a))}}^\circ {e^{(b \log(a))}}^\downarrow.$$ We must study the coefficients of the Ressayre form of $a^b$. 
	Note that ${e^{(b\log(a))}}^\uparrow$ is a monomial, so it does not contribute to the coefficients. Let us consider the other two factors. 
	
	Write $a = \sum_{i<\alpha} e^{\gamma_i} a_i = e^{\gamma_0} a_0 (1 + \eps)$ where $\eps = \sum_{1\leq i <\alpha} \frac{a_i}{a_0} e^{\gamma_i - \gamma_0}$. Then $$\log(a) = \gamma_0 + \log(a_0) + \log(1+\eps).$$  Since $b$ is purely infinite, ${(b \log(a))}^\circ = (b \log(1+\eps))^\circ$. Since $\log(1+\eps)$ is a power series in $\eps$ with rational coefficients, and the  
	coefficients $\frac{a_i}{a_0}$ of $\eps$ belong to $\E$, it follows that ${(b \log(a))}^\circ\in \E$ and therefore ${e^{(b \log(a))}}^\circ \in \E^+$. This is the leading coefficient of $a^b$. 
	
	The other coefficients of $a^b$ come from the power series expansion of ${e^{(b \log(a))}}^\downarrow$, so they belong to the ring generated by the coefficients of $b$ and those of $\eps$, which is included in $\E$. 
\end{proof}

\subsection*{Acknowledgements} The first author gave a plenary talk at the XXI congress of the Unione Matematica Italiana, Pavia 2-7 September 2019 where some of the results of this paper were also mentioned. We thank Vincenzo Mantova and an anonymous referee for their comments on a preliminary version which helped to improve the exposition.

\def\Url#1{\href{https://dx.doi.org/#1}{\textrm{#1}}\endgroup}

\bibliographystyle{plainnat}


\begin{thebibliography}{25}
	\providecommand{\natexlab}[1]{#1}
	\providecommand{\url}[1]{\texttt{#1}}
	\expandafter\ifx\csname urlstyle\endcsname\relax
	\providecommand{\doi}[1]{doi: #1}\else
	\providecommand{\doi}{doi: \begingroup \urlstyle{rm}\Url}\fi
	
	\bibitem[Altman(2017)]{Altman2017}
	Harry~J. Altman.
	\newblock {Intermediate arithmetic operations on ordinal numbers}.
	\newblock \emph{Mathematical Logic Quarterly}, 63\penalty0 (3-4):\penalty0
	228--242, nov 2017. 
	\newblock \doi{10.1002/malq.201600006}. 
	
	\bibitem[Aschenbrenner et~al.(2017)Aschenbrenner, van~den Dries, and van~der
	Hoeven]{Aschenbrenner2015}
	Matthias Aschenbrenner, Lou van~den Dries, and Joris van~der Hoeven.
	\newblock \emph{{Asymptotic Differential Algebra and Model Theory of
			Transseries}}.
	\newblock Annals of Mathematical Studies. Princeton University Press,
	Princeton, dec 2017. 
	\newblock \doi{10.1515/9781400885411}. 
	
	\bibitem[Berarducci and Mantova(2018)]{Berarducci2018}
	Alessandro Berarducci and Vincenzo Mantova.
	\newblock {Surreal numbers, derivations and transseries}.
	\newblock \emph{Journal of the European Mathematical Society}, 20\penalty0
	(2):\penalty0 339--390, jan 2018. 
	\newblock \doi{10.4171/JEMS/769}. 
	
	\bibitem[Berarducci and Mantova(2019)]{Berarducci2019}
	Alessandro Berarducci and Vincenzo Mantova.
	\newblock {Transseries as germs of surreal functions}.
	\newblock \emph{Transactions of the American Mathematical Society},
	371\penalty0 (5):\penalty0 3549--3592, 2019.
	\newblock \doi{10.1090/tran/7428}. 
	
	\bibitem[Berarducci et~al.(2018)Berarducci, Kuhlmann, Mantova, and
	Matusinski]{Berarducci2018b}
	Alessandro Berarducci, Salma Kuhlmann, Vincenzo Mantova, and Micka{\"{e}}l
	Matusinski.
	\newblock {Exponential fields and Conway's omega-map}.
	\newblock \emph{Proc. Amer. Math. Soc, to appear}, pages 1--15, 2018. 
	
	\bibitem[Carruth(1942)]{Carruth1942}
	Philip~W. Carruth.
	\newblock {Arithmetic of ordinals with applications to the theory of ordered
		abelian groups}.
	\newblock \emph{Bulletin of the American Mathematical Society}, 48\penalty0
	(4):\penalty0 262--271, 1942. 
	\newblock \doi{10.1090/S0002-9904-1942-07649-X}.
	
	\bibitem[Conway(1976)]{Conway76}
	John~H. Conway.
	\newblock \emph{{On number and games}}, volume~6 of \emph{London Mathematical
		Society Monographs}.
	\newblock Academic Press, London, 1976. 
	
	\bibitem[Dahn(1984)]{Dahn1984a}
	Bernd~I. Dahn.
	\newblock {The limit behaviour of exponential terms}.
	\newblock \emph{Fundamenta Mathematicae}, 124\penalty0 (2):\penalty0 169--186,
	1984.
	
		\bibitem[van~den Dries and Ehrlich(2001)]{DriesE2001}
		Lou van~den Dries and Philip Ehrlich.
		\newblock {Fields of surreal numbers and exponentiation}.
		\newblock \emph{Fundamenta Mathematicae}, 167\penalty0 (2):\penalty0 173--188,
		2001, 
		\newblock \doi{10.4064/fm167-2-3}, 
		\newblock et Erratum, \emph{Fundamenta mathematicae}, 168 (2): 295-297.
		
		\bibitem[van~den Dries and Levitz(1984)]{Dries1984}
		Lou van~den Dries and Hilbert Levitz.
		\newblock {On Skolem's exponential functions below $2^{2^x}$}.
		\newblock \emph{Transactions of the American Mathematical Society},
		286\penalty0 (1):\penalty0 339--349, 1984.
		
		\bibitem[van~den Dries et~al.(2001)van~den Dries, Macintyre, and
		Marker]{DriesMM2001}
		Lou van~den Dries, Angus Macintyre, and David Marker.
		\newblock {Logarithmic-exponential series}.
		\newblock \emph{Annals of Pure and Applied Logic}, 111\penalty0 (1-2):\penalty0
		61--113, jul 2001.
		\newblock \doi{10.1016/S0168-0072(01)00035-5}.
		
	\bibitem[de~Jongh and Parikh(1977)]{DeJongh1977}
	D.H.J de~Jongh and Rohit Parikh.
	\newblock {Well-partial orderings and hierarchies}.
	\newblock \emph{Indagationes Mathematicae (Proceedings)}, 80\penalty0
	(3):\penalty0 195--207, 1977.
	\newblock \doi{10.1016/1385-7258(77)90067-1}.
 
	\bibitem[Ehrenfeucht(1973)]{Ehrenfeucht73}
	Andrzej Ehrenfeucht.
	\newblock {Polynomial functions with exponentiation are well ordered}.
	\newblock \emph{algebra universalis}, 3\penalty0 (1):\penalty0 261--262, 1973.
	\newblock \doi{10.1007/BF02945125}.
 
	
	\bibitem[Gonshor(1986)]{Gonshor1986}
	Harry Gonshor.
	\newblock \emph{{An introduction to the theory of surreal numbers}}.
	\newblock London Mathematical Society Lecture Notes Series. Cambridge
	University Press, Cambridge, 1986.
	\newblock \doi{10.1017/CBO9780511629143}.
 
	
	\bibitem[Gurevi{\v{c}}(1986)]{Gurevlc1986}
	Reuben Gurevi{\v{c}}.
	\newblock {Transcendental Numbers and Eventual Dominance of Exponential
		Functions}.
	\newblock \emph{Bulletin of the London Mathematical Society}, 18\penalty0
	(6):\penalty0 560--570, nov 1986.
	\newblock \doi{10.1112/blms/18.6.560}. 
	
	\bibitem[Hahn(1907)]{Hahn1907}
	Hans Hahn.
	\newblock {{\"{U}}ber die nichtarchimedischen Gr{\"{o}}ssenszsteme}.
	\newblock \emph{Sitz. der K. Akad der Wiss., Math Nat. KL.}, 116\penalty0
	(IIa):\penalty0 601--655, 1907.
	
	\bibitem[Hardy(1910)]{Hardy1910}
	G.~H. Hardy.
	\newblock \emph{{Orders of infinity,The `infinit\"arcalc\"ul' of Paul du
			Bois-Reymond}}.
	\newblock Cambridge University Press, 1910.
	
	\bibitem[Kruskal(1960)]{Kruskal1960}
	J.~B. Kruskal.
	\newblock {Well-Quasi-Ordering, The Tree Theorem, and Vazsonyi's Conjecture}.
	\newblock \emph{Transactions of the American Mathematical Society}, 95\penalty0
	(2):\penalty0 210, may 1960.
	\newblock \doi{10.2307/1993287}. 
	
	\bibitem[Levitz(1978)]{Levitz1978}
	Hilbert Levitz.
	\newblock {An ordinal bound for the set of polynomial functions with
		exponentiation}.
	\newblock \emph{Algebra Universalis}, 8\penalty0 (1):\penalty0 233--243, dec
	1978.
	\newblock \doi{10.1007/BF02485393}.
	
	\bibitem[Lipparini(2018)]{Lipparini2018}
	Paolo Lipparini.
	\newblock {Some transfinite natural sums}.
	\newblock \emph{Mathematical Logic Quarterly}, 64\penalty0 (6):\penalty0
	514--528, 2018.
	\newblock \doi{10.1002/malq.201600092}.
	
	\bibitem[Neumann(1949)]{Neumann1949}
	Bernhard~Hermann Neumann.
	\newblock {On ordered division rings}.
	\newblock \emph{Trans. Amer. Math. Soc}, 66\penalty0 (1):\penalty0 202--252,
	1949.
	
	\bibitem[Richardson(1969)]{Richardson1969}
	Daniel Richardson.
	\newblock {Solution of the identity problem for integral exponential
		functions}.
	\newblock \emph{Zeitschrift f{\"{u}}r Mathematische Logik und Grundlagen der
		Mathematik}, 15\penalty0 (20-22):\penalty0 333--340, 1969.
	\newblock \doi{10.1002/malq.19690152003}. 
	
	\bibitem[Schmidt(1978)]{Schmidt1978}
	Diana Schmidt.
	\newblock {Associative Ordinal Functions, Well Partial Orderings and a Problem
		of Skolem}.
	\newblock \emph{Zeitschrift f{\"{u}}r Mathematische Logik und Grundlagen der
		Mathematik}, 24\penalty0 (19-24):\penalty0 297--302, 1978.
	\newblock \doi{10.1002/malq.19780241904}.
	
	\bibitem[Sierpinski(1958)]{Sierpinski1958}
	Wac{\l}aw Sierpi\'nski.
	\newblock Cardinal and ordinal numbers. 
	\newblock \emph{Polska Akademia Nauk, Monografie Matematyczne}. Tom 34 Pa\'nstwowe Wydawnictwo Naukowe, Warsaw 1958, 487 pp.
	
	\bibitem[Skolem(1956)]{Skolem1956}
	Thoralf Skolem.
	\newblock {An ordered set of arithmetic functions representing the least
		{$\epsilon$}-number}.
	\newblock \emph{Norske Vid. Selsk. Forh., Trondheim}, 29:\penalty0 54--59, 1956.
	
	\bibitem[Wilkie(1996)]{Wilkie1996}
	Alex~J. Wilkie.
	\newblock {Model completeness results for expansions of the ordered field of
		real numbers by restricted Pfaffian functions and the exponential function}.
	\newblock \emph{J. Amer. Math. Soc}, 9\penalty0 (4):\penalty0 1051--1094, 1996.
	
\end{thebibliography}

\end{document}